\documentclass[11pt,twoside,a4paper]{article}
\usepackage[T1]{fontenc}
\usepackage[utf8]{inputenc}

\usepackage[australian]{babel}

\usepackage[driver=pdftex,margin=3cm,heightrounded=true,centering,includeheadfoot]{geometry}

\usepackage{hyperref}

\usepackage[intlimits,tbtags]{mathtools}
\usepackage{amssymb,amsfonts}
\usepackage{amsthm}

\usepackage{thmtools,thm-restate}

\usepackage[capitalise]{cleveref}

\allowdisplaybreaks
\numberwithin{equation}{section}

\usepackage[dvipsnames]{xcolor}
\definecolor{MyBlue}{cmyk}{1,0.13,0,0.63}
\definecolor{MyGreen}{cmyk}{0.91,0,0.88,0.52}
\definecolor{MyRed}{rgb}{.6,0,0}
\newcommand{\mylinkcolor}{MyBlue}
\newcommand{\mycitecolor}{MyGreen}
\newcommand{\myurlcolor}{MyRed}

\makeatletter
\def\@endtheorem{\endtrivlist}
\makeatother
\theoremstyle{plain}
\newtheorem{thm}{Theorem}[section]

\newtheorem*{main*}{Main Theorem}
\newtheorem{lem}[thm]{Lemma}
\newtheorem{prop}[thm]{Proposition}
\newtheorem{coro}[thm]{Corollary}
\theoremstyle{definition}
\newtheorem{defn}[thm]{Definition}
\newtheorem{remark}[thm]{Remark}
\newtheoremstyle{myAssumption}
  {\topsep}
  {\topsep}
  {\normalfont}
  {0pt}
  {\bfseries}
  {.}
  {5pt plus 1pt minus 1pt}
  {\thmname{#1}\thmnote{ #3}}
\theoremstyle{myAssumption}
\newtheorem*{assumption*}{Assumption}

\makeatletter
\newcommand{\customlabel}[2]{%
   \protected@write \@auxout {}{\string \newlabel {#1}{{#2}{\thepage}{#2}{#1}{}} }%
   \hypertarget{#1}{}
}
\makeatother

\renewcommand{\eqref}[1]{\labelcref{#1}}
\crefname{thm}{Theorem}{Theorems}
\crefname{lem}{Lemma}{Lemmas}
\crefname{prop}{Proposition}{Propositions}
\crefname{coro}{Corollary}{Corollaries}
\crefname{defn}{Definition}{Definitions}
\crefname{remark}{Remark}{Remarks}

\hypersetup{%
  bookmarksnumbered=true,bookmarksopen=false,%
  plainpages=false,%
  linktocpage=true,%
  colorlinks=true,breaklinks=true,%
  linkcolor=\mylinkcolor,citecolor=\mycitecolor,urlcolor=\myurlcolor,%
  pdfpagelayout=OneColumn,%
  pageanchor=true,%
}

\makeatletter
\def\thm@space@setup{%
  \thm@preskip=4pt plus 2pt minus 2pt
  \thm@postskip=\thm@preskip
}
\renewenvironment{proof}[1][\proofname]{\par
  \pushQED{\qed}%
  \normalfont \topsep4\p@\relax
  \trivlist
  \item[\hskip\labelsep
        \itshape
    #1\@addpunct{.}]\ignorespaces
}{%
  \popQED\endtrivlist\@endpefalse
}
\makeatother

\usepackage[shortlabels]{enumitem}
\setlist{topsep=4pt plus 2pt minus 2pt,partopsep=0pt,itemsep=2pt plus 2pt minus 2pt,parsep=0.5\parskip}
\setenumerate[1]{label=(\arabic*),font=\upshape}

\usepackage{fancyhdr}
\newcommand{\MR}[1]{}

\let\OLDthebibliography\thebibliography
\renewcommand\thebibliography[1]{
  \addcontentsline{toc}{section}{\refname}
  \OLDthebibliography{#1}
  \setlength{\parskip}{0pt}
  \setlength{\itemsep}{0pt plus 0.3ex}
}

\newcommand{\N}{\mathbb{N}}
\newcommand{\R}{\mathbb{R}}
\newcommand{\C}{\mathbb{C}}
\newcommand{\Z}{\mathbb{Z}}

\newcommand{\mH}{\mathcal{H}}
\newcommand{\D}{\mathcal{D}}
\newcommand{\E}{\mathcal{E}}
\newcommand{\A}{\mathcal{A}}
\newcommand{\B}{\mathcal{B}}
\newcommand{\mK}{\mathcal{K}}
\newcommand{\mL}{\mathcal{L}}
\newcommand{\mO}{\mathcal{O}}
\newcommand{\mT}{\mathcal{T}}

\newcommand{\pS}{\mathcal{S}}
\newcommand{\pT}{\mathcal{T}}

\renewcommand{\Re}{\mathop{\textnormal{Re}}}

\renewcommand{\bar}[1]{\overline{#1}}

\DeclareMathOperator{\Dom}{Dom}
\DeclareMathOperator{\Ran}{Ran}
\DeclareMathOperator{\End}{End}

\DeclareMathOperator{\supp}{supp}
\DeclareMathOperator{\Index}{Index}
\DeclareMathOperator{\ind}{ind}
\DeclareMathOperator{\relind}{rel-ind}
\DeclareMathOperator{\ch}{ch}
\DeclareMathOperator{\spec}{spec}

\newcommand{\dvol}{\textnormal{dvol}}
\newcommand{\K}{K}
\newcommand*{\KK}{\relax\ifmmode{K\!K}\else{\itshape K\hspace{-.07em}K}\fi}
\newcommand{\SF}{\textnormal{sf}}

\newcommand{\til}[1]{\widetilde{#1}}
\newcommand{\la}{\langle}
\newcommand{\ra}{\rangle}
\newcommand{\into}{\hookrightarrow}

\newcommand{\bigmvert}{\,\big|\,}

\newcommand{\bundlefont}[1]{{\mathtt{#1}}}
\newcommand{\bF}{\bundlefont{F}}

\newcommand{\dx}[1][x]{\, \mathrm{d}#1}

\newcommand{\mattwo}[4]{
  \begin{pmatrix}#1&#2\\ #3&#4\end{pmatrix}
}

\makeatletter
\DeclareFontFamily{OMX}{MnSymbolE}{}
\DeclareSymbolFont{MnLargeSymbols}{OMX}{MnSymbolE}{m}{n}
\SetSymbolFont{MnLargeSymbols}{bold}{OMX}{MnSymbolE}{b}{n}
\DeclareFontShape{OMX}{MnSymbolE}{m}{n}{
    <-6>  MnSymbolE5
   <6-7>  MnSymbolE6
   <7-8>  MnSymbolE7
   <8-9>  MnSymbolE8
   <9-10> MnSymbolE9
  <10-12> MnSymbolE10
  <12->   MnSymbolE12
}{}
\DeclareFontShape{OMX}{MnSymbolE}{b}{n}{
    <-6>  MnSymbolE-Bold5
   <6-7>  MnSymbolE-Bold6
   <7-8>  MnSymbolE-Bold7
   <8-9>  MnSymbolE-Bold8
   <9-10> MnSymbolE-Bold9
  <10-12> MnSymbolE-Bold10
  <12->   MnSymbolE-Bold12
}{}

\let\llangle\@undefined
\let\rrangle\@undefined
\DeclareMathDelimiter{\llangle}{\mathopen}%
                     {MnLargeSymbols}{'164}{MnLargeSymbols}{'164}
\DeclareMathDelimiter{\rrangle}{\mathclose}%
                     {MnLargeSymbols}{'171}{MnLargeSymbols}{'171}
\makeatother

\newcommand{\lla}{\llangle}
\newcommand{\rra}{\rrangle}

\def\myTitle{Generalised Dirac--Schrödinger operators and the Callias Theorem}

\title{Generalised Dirac--Schrödinger operators \\ and the Callias Theorem}
\author{
Koen van den Dungen%
\footnote{Email: \texttt{kdungen@uni-bonn.de}}
\\[2mm]
{\small Mathematisches Institut}, 
{\small Universität Bonn}\\
{\small Endenicher Allee 60, D-53115 Bonn}
}

\date{}

\begin{document}

\maketitle

\begin{abstract}
\noindent
We consider generalised Dirac--Schrödinger operators, consisting of a self-adjoint elliptic first-order differential operator $\D$ with a skew-adjoint `potential' given by a (suitable) family of unbounded operators. 
The index of such an operator represents the pairing (Kasparov product) of the $\K$-theory class of the potential with the $\K$-homology class of $\D$. 
Our main result in this paper is a generalisation of the Callias Theorem: the index of the Dirac--Schrödinger operator can be computed on a suitable compact hypersurface. 
Our theorem simultaneously generalises (and is inspired by) the well-known result that the spectral flow of a path of relatively compact perturbations depends only on the endpoints. 

\vspace{\baselineskip}

\end{abstract}

\section{Introduction}

Let $\D_M$ be the Dirac operator on an odd-dimensional locally compact smooth spin manifold $M$, which determines the $K$-homology fundamental class $[\D_M] \in \K_1(M)$. 
If $\Sigma_M \in \K^1(M)$ is any element in the $\K$-theory of $M$, there is a natural pairing $\langle~,~\rangle \colon \K^1(M) \times \K_1(M) \to \Z$ of $\K$-theory with $\K$-homology, often referred to as the \emph{index pairing}, which yields an integer  
\[
\big\langle \Sigma_M , [\D_M] \big\rangle \in \Z .
\]
We now wish to compute this integer in terms of an index pairing on a suitable hypersurface $N \subset M$ of codimension one. 
Assume that $U\subset M$ is an open subset with compact closure $K = \overline{U}$ and smooth boundary $N = \partial U$. 
Consider the natural map ${\iota_U}^* \colon \K_1(M) \to \K_1(U)$ which sends $[\D_M]$ to the class $[\D_U]$ of the restriction of the Dirac operator to $U$. 
The $\K$-homology boundary map 
\[
\partial \colon \K_1(U) \to \K_0(\partial U)
\]
sends the class $[\D_U]$ to the class $[\D_{\partial U}]$ of the Dirac operator on the boundary $\partial U$. 

Now suppose that the $\K$-theory class $\Sigma_M$ is the image of a class $\Sigma_U \in \K^1(U)$ under the natural map ${\iota_U}_* \colon \K^1(U) \to \K^1(M)$. 
Suppose furthermore that $\Sigma_U$ is the image of a class $\Sigma_{\partial U} \in \K^0(\partial U)$ under the $\K$-theory boundary map $\partial \colon \K^0(\partial U) \to \K^1(U)$. 
Then by naturality we have the equalities 
\begin{align*}
\big\langle \Sigma_M , [\D_M] \big\rangle 
&= \big\langle {\iota_U}_*\circ\partial(\Sigma_{\partial U}) , [\D_M] \big\rangle 
= \big\langle \partial(\Sigma_{\partial U}) , {\iota_U}_*([\D_M]) \big\rangle \\
&= \big\langle \Sigma_{\partial U} , \partial\circ{\iota_U}^*([\D_M]) \big\rangle 
= \big\langle \Sigma_{\partial U} , [\D_{\partial U}] \big\rangle .
\end{align*}

To summarise the preceding argument, if there exists a class $\Sigma_{\partial U} \in \K^1(\partial U)$ with $\Sigma_M = {\iota_U}_*\circ\partial(\Sigma_{\partial U})$, then the index pairing on the locally compact manifold $M$ can be computed from an index pairing on the compact hypersurface $\partial U$:
\begin{equation}
\label{eq:index_pairings}
\big\langle \Sigma_M , [\D_M] \big\rangle = \big\langle \Sigma_{\partial U} , [\D_{\partial U}] \big\rangle .
\end{equation}

There are two (at first sight rather different) instances in the literature where such a computation has been made. 
The first instance is in the case of \emph{Dirac--Schrödinger} (or \emph{Callias-type}) operators. 
These are operators of the form $\D_M-i\pS$, where the `potential' $\pS$ is a self-adjoint endomorphism on some auxiliary vector bundle (of finite rank) over $M$. Assuming that $\pS$ is invertible outside of $U\subset M$, the Dirac--Schrödinger operator $\D_M-i\pS$ is Fredholm and its index computes the index pairing of the $\K$-theory class of the potential with the $\K$-homology class of the Dirac operator: 
\[
\Index(\D_M-i\pS) = \big\langle [\pS] , [\D_M] \big\rangle \in \Z .
\]
The invertibility of $\pS$ outside $U$ ensures that $[\pS] = {\iota_U}_*([\pS|_U])$. Moreover, since we may perturb the potential on a compact subset without changing its $\K$-theory class, and since $\overline{U}$ is compact, we note that $j_*([\pS|_U]) = [\pS|_{\overline{U}}] = 0 \in \K^1(\overline{U})$, where $j_* \colon \K^1(U) \to \K^1(\overline{U})$ is induced from the inclusion $U\into\overline{U}$. By exactness of the sequence 
\[
\K^0(\partial U) \xrightarrow{\partial} \K^1(U) \xrightarrow{j_*} \K^1(\overline{U}) ,
\]
we therefore know that $[\pS|_U] = \partial(\Sigma_{\partial U})$ for some $\Sigma_{\partial U} \in \K^0(\partial U)$. 
In fact, as we will see in \cref{coro:classical_Callias}, we can explicitly identify $\Sigma_{\partial U}$ as the $\K$-theory class of the vector bundle $V_+$ over $\partial U$ obtained from the positive eigenspace of the invertible self-adjoint endomorphism $\pS|_{\partial U}$. 
The pairing $\langle [V_+] , [\D_{\partial U}] \rangle$ can be computed as the index of the operator $(\D_{\partial U})^+_+$, which is obtained by twisting the chiral Dirac operator on $\partial U$ with the vector bundle $V_+$. 
The equality \eqref{eq:index_pairings} therefore yields
\begin{equation}
\label{eq:Callias_Thm}
\Index(\D_M-i\pS) = \Index(\D_{\partial U})^+_+ .
\end{equation}
This result is known as the \emph{Callias Theorem}; the first version was proven by Callias \cite{Cal78} on Euclidean space and it has subsequently been generalised by various authors (see, for instance, \cite{Ang90,BM92,Ang93a,Rad94,Bun95,Kuc01,Gesztesy-Waurick16}). 

The second instance of \cref{eq:index_pairings} appears in the study of \emph{spectral flow}. 
Consider a `sufficiently continuous' family of self-adjoint Fredholm operators $\{\pS(x)\}_{x\in[0,1]}$ with invertible endpoints and with common domain on a Hilbert space $\mH$, such that $\pS(x)$ is a relatively compact perturbation of $\pS(0)$ (for each $x\in[0,1]$).
Then the spectral flow depends only on the endpoints and is given by (see \cite[Theorem 3.6]{Les05} and \cite[Proposition 2.5]{Wah08})
\begin{align}
\label{eq:sf_rel-ind}
\SF\big(\{\pS(x)\}_{x\in[0,1]}\big) &= \relind\big(P_+(\pS(1)),P_+(\pS(0))\big) .
\end{align}
Here the left-hand-side is the spectral flow of the family $\{\pS(x)\}_{x\in[0,1]}$, 
and the right-hand-side is given by the relative index of the pair of positive spectral projections associated to $\pS(1)$ and $\pS(0)$. 
To view this equality in the form of \cref{eq:index_pairings}, let $M=\R$, $U=(0,1)$, and $N=\partial U=\{0\}\cup\{1\}$, and extend $\pS$ to a family on $\R$. 
By the well-known `index = spectral flow' equality of Robbin--Salamon (see e.g.\ \cite{RS95,Wah07,AW11} and \cite[Corollary 5.16]{vdD19_Index_DS}), the spectral flow can be described as an index pairing on $\R$:
\begin{align*}
\SF\big( \{\pS(x)\}_{x\in\R} \big) 
&= \Index\big( \partial_x+\pS(\cdot) \big) 
= \big\langle [\pS(\cdot)] , [-i\partial_x] \big\rangle ,
\end{align*}
where $-i\partial_x$ is the standard Dirac operator on $\R$. 
Moreover, the assumption that $\pS(x)$ is a relatively compact perturbation of $\pS(0)$ (for each $x\in[0,1]$), combined with the compactness of $[0,1]$, implies that the operator $\pS(\cdot)|_{[0,1]}$ is a relatively compact perturbation of a constant invertible family. 
It follows that $j_*([\pS(\cdot)|_{(0,1)}]) = 0 \in \K^1([0,1])$ (where $j_*$ is induced from the inclusion $(0,1)\into[0,1]$) and therefore (by exactness, as in the case of Dirac--Schrödinger operators) $[\pS(\cdot)|_{(0,1)}] = \partial\big(\Sigma_{\{0\}\cup\{1\}}\big)$ for some element $\Sigma_{\{0\}\cup\{1\}} \in \K^0(\{0\}\cup\{1\}) \simeq \Z\oplus\Z$. 
We will see in \cref{coro:Callias_sf} that the relative index on the right-hand-side of \cref{eq:sf_rel-ind} is indeed obtained from an index pairing of $\Sigma_{\{0\}\cup\{1\}}$ with the $\K$-homology element $[\D_{\{0\}\cup\{1\}}] \in \K_0(\{0\}\cup\{1\})$, where the latter can be identified with $(-1)\oplus1\in\Z\oplus\Z$. 

Thus we have seen that both the Callias Theorem \eqref{eq:Callias_Thm} and the spectral flow result \eqref{eq:sf_rel-ind} can be viewed as a special case of the equality \eqref{eq:index_pairings}. 
Our goal in the present paper is to provide a common generalisation which unifies both these results. 
For this purpose, we now consider `generalised' Dirac--Schrödinger operators, which are again operators of the form $\D_M-i\pS(\cdot)$, where now the auxiliary vector bundle is of \emph{infinite} rank, and the `potential' $\pS(\cdot)$ consists of a family of (unbounded) self-adjoint operators $\{\pS(x)\}_{x\in M}$ on a fixed Hilbert space $\mH$.
(In fact, instead of $\mH$ we will more generally consider a Hilbert $C^*$-module over some auxiliary $C^*$-algebra, but in this introduction we limit our attention to the simpler case of a Hilbert space.)
Such operators were studied in \cite[\S8]{KL13} for suitably differentiable potentials, and in \cite{vdD19_Index_DS} for continuous potentials. 
It is known that the pairing of the $\K$-theory class of $\pS(\cdot)$ with the $\K$-homology class of $\D_M$ still equals the index of the Dirac--Schrödinger operator (\cite[Theorem 1.2]{KL13} and \cite[Theorem 5.15]{vdD19_Index_DS}): 
\[
\big\langle [\pS(\cdot)] , [\D_M] \big\rangle = \Index\left( \D-i\pS(\cdot) \right) .
\]
Now, under the additional assumption that the potential $\pS(\cdot)$ is given by a family of relatively compact perturbations (as in the case of the spectral flow result \eqref{eq:sf_rel-ind}), we again find that $[\pS(\cdot)] = {\iota_U}_*\circ\partial(\Sigma_{\partial U})$ for some $\K$-theory class $\Sigma_{\partial U} \in \K^0(\partial U)$.  
Hence \cref{eq:index_pairings} applies, and to obtain our desired generalisation of the Callias Theorem, it remains only to identify the class $\Sigma_{\partial U}$. 
We will see that this class can again (as in the spectral flow case) be described as a relative index of positive spectral projections, and our generalised Callias Theorem then provides the equality (we refer to \S\ref{subsec:Callias} for the precise statement)
\begin{align*}
\Index\big(\D-i\pS(\cdot)\big) 
= \big\langle \relind\big(P_+(\pS_N(\cdot)),P_+(\pT(\cdot))\big) , [\D_N] \big\rangle .
\end{align*}

Let us provide a brief summary of the contents of this paper. 
We start in \cref{sec:background} with some background material, describing both the classical Callias Theorem as well as the study of the spectral flow for relatively compact perturbations (in particular, we prove \cref{eq:sf_rel-ind}). 
In \cref{sec:Callias} we describe our main assumptions, definitions, and the main results (the proofs are postponed to later sections). In particular, we state in \S\ref{subsec:Callias} our generalisation of the classical Callias Theorem, and show in \S\ref{subsec:special_cases} that we recover the Callias Theorem \eqref{eq:Callias_Thm} and the spectral flow result \eqref{eq:sf_rel-ind}. 
The remaining sections are devoted to the proofs. 
In \cref{sec:DS} we first consider generalised Dirac--Schrödinger operators. 
As an important tool, we present a relative index theorem for such operators in \S\ref{sec:rel_index_thm}, under fairly general assumptions. \S\ref{sec:Fred_index} and \S\ref{sec:Kasp_prod} then provide sufficient conditions under which we can prove that a Dirac--Schrödinger operator is Fredholm, and that its index equals the pairing between the $\K$-theory and $\K$-homology classes. 
\cref{sec:proof} finally provides the proof of the generalised Callias-type theorem. 

Throughout the body of this article, we will work with Hilbert $C^*$-modules over $C^*$-algebras (rather than just Hilbert spaces). 
An important step in our main result is to ensure that the relative index is well-defined, for which we require several operator-theoretic facts that are known for operators on Hilbert spaces, but (to the author's best knowledge) have not yet appeared in the literature for operators on Hilbert $C^*$-modules. 
We therefore include an Appendix, in which we generalise several (partly well-known) operator-theoretic results from Hilbert spaces to Hilbert $C^*$-modules. 
We mention here a few of these results, which may also be of independent interest:
\begin{itemize}
\item 
The composition of a $*$-strongly convergent sequence of adjointable operators with a compact operator yields a norm-convergent sequence (\cref{lem:s-limit_times_cpt}). 
\item 
Let $T$ be regular and self-adjoint. Then any relatively $T$-compact operator $R$ is relatively $T$-bounded with arbitrarily small relative bound (\cref{prop:rel_cpt_rel_bdd_0}). Consequently, if $R$ is symmetric, then $T+R$ is again regular and self-adjoint on $\Dom(T)$ (\cref{prop:rel_cpt_pert_reg_sa}). 
\item 
Let $T$ be regular and self-adjoint, and let $R$ be symmetric and relatively $T$-compact. 
Let $f\in C(\R)$ be a continuous function for which the limits $\lim_{x\to\pm\infty}f(x)$ exist. 
Then $f(T+R)-f(T)$ is compact (\cref{prop:cpt_diff_cts_funct}). 
Furthermore, if $T$ and $T+R$ are both invertible, also $P_+(T+R)-P_+(T)$ is compact (where $P_+(T)$ denotes the positive spectral projection of $T$). 
\end{itemize}

\subsection{Notation}

Throughout this paper, let $A$ be a $\sigma$-unital $C^*$-algebra, and let $E$ be a (possibly $\Z_2$-graded) countably generated Hilbert $C^*$-module over $A$ (or Hilbert $A$-module for short) with $A$-valued inner product $\la\cdot|\cdot\ra$. 
(The reader unfamiliar with $C^*$-modules may consider the special case $A=\C$, so that $E$ is simply a separable Hilbert space. For an introduction to Hilbert $C^*$-modules, we refer to \cite{Lance95}.) 
For the inner product of an element $\psi\in E$ with itself we use the convenient short-hand notation
\[
\lla \psi \rra := \la \psi | \psi \ra .
\]
The norm of $\psi$ is then given by $\|\psi\| = \|\lla\psi\rra\|^{\frac12}$. 

The space of adjointable linear operators $E\to E$ is denoted by $\mL_A(E)$.
For any $\psi,\eta\in E$, the rank-one operators $\theta_{\psi,\eta}$ are defined by $\theta_{\psi,\eta}\xi := \psi \la\eta|\xi\ra$ for $\xi\in E$. 
The compact operators $\mK_A(E)$ are given by the closure of the space of finite linear combinations of rank-one operators. 
For two Hilbert $A$-modules $E_1$ and $E_2$, the adjointable linear operators $E_1\to E_2$ are denoted by $\mL_A(E_1,E_2)$. 

A densely defined operator $S$ is called \emph{regular} if $S$ is closed, the adjoint $S^*$ is densely defined, and $1+S^*S$ has dense range (note that on a Hilbert space, every closed operator is regular). 
A densely defined, closed, symmetric operator $S$ is regular and self-adjoint if and only if the operators $S\pm i$ are surjective \cite[Lemma 9.8]{Lance95}. 

Given a densely defined, symmetric operator $S$ on $E$, we can equip $\Dom S$ with the graph inner product $\la\psi|\psi\ra_S := \la(S\pm i)\psi|(S\pm i)\psi\ra = \la\psi|\psi\ra + \la S\psi|S\psi\ra$. The graph norm of $S$ is then defined as $\|\psi\|_S := \big\|\la\psi|\psi\ra_S\big\|^{\frac12} = \|(S\pm i)\psi\|$.

\section{Background}
\label{sec:background}

As described in the Introduction, our main theorem simultaneously generalises both the classical Callias Theorem \eqref{eq:Callias_Thm} and the spectral flow equality \ref{eq:sf_rel-ind}. 
In this section, we will introduce both these results. 

\subsection{The classical Callias Theorem}

The `classical Callias Theorem', which we aim to generalise, is a result which was first proven by Callias \cite{Cal78} on Euclidean space, who proved that the index of a Dirac--Schrödinger operator $\D-i\pS$ on $\R^{2n+1}$ can be computed on a sufficiently large sphere. 
This has become known as the Callias Theorem and has been generalised by various authors (see, for instance, \cite{Ang90,BM92,Ang93a,Rad94,Bun95,Kuc01,Gesztesy-Waurick16}), replacing Euclidean space by larger classes of Riemannian manifolds, and computing the index on a suitable hypersurface (which is the boundary of a compact subset). 
The Callias Theorem continues to be actively studied, 
with recent work considering for instance Callias-type operators on Lie manifolds \cite{CN14}, with degenerate potentials \cite{Kot15}, via cobordism invariance \cite{BS16}, on manifolds with boundary \cite{Shi17}, twisted by Hilbert $C^*$-bundles of finite type \cite{Cec20}, and associated to abstract spectral triples \cite{SS23}. 

We will cite here Anghel's version \cite{Ang93a} of the Callias Theorem. 
Let $M$ be a complete odd-dimensional oriented Riemannian manifold, and let $\D$ be a formally self-adjoint Dirac-type operator on a hermitian Clifford bundle $\Sigma\to M$. 
Let $\pS = \pS^* \in \Gamma^\infty(\End\Sigma)$ be a hermitian bundle endomorphism such that $\pS$ commutes with Clifford multiplication, $\pS$ and $[\D,\pS]$ are uniformly bounded, and there exists a compact subset $K\subset M$ such that $\pS$ is uniformly invertible on the complement of $K$. 

Without loss of generality, assume that $K$ has a smooth compact boundary $N$. 
On $\Sigma_N := \Sigma|_N$, we have a $\Z_2$-grading operator $\Gamma_N$ given by Clifford multiplication with the unit normal vector on $N$, which yields the decomposition $\Sigma_N = \Sigma_N^+\oplus\Sigma_N^-$. 
Consider a `restriction' $\D_N$ of $\D$ to $\Sigma_N$ (i.e., the principal symbol of $\D_N$ is obtained from the principal symbol of $\D$ by restricting from $TM$ to $TN$), which anticommutes with $\Gamma_N$. 
Since $\D_N$ is elliptic and $N$ is compact, $\D_N$ has compact resolvents. In particular, $\D_N$ is Fredholm, and we obtain a $\K$-homology class $[\D_N] \in \K^0(C(N)) \equiv \K_0(N)$. 

Let $\pS_N$ denote the restriction of $\pS$ to $\Sigma_N\to N$. 
We define $\Sigma_{N+} := \Ran P_+(\pS_N)$ to be the image of the positive spectral projection of $\pS_N$, representing a $\K$-theory class $[\Sigma_{N+}] \in \K_0(C(N)) \equiv \K^0(N)$. 

Since $\pS$ commutes with the Clifford multiplication, $\Gamma_N$ is still a $\Z_2$-grading on $\Sigma_{N+}$ and yields the decomposition $\Sigma_{N+} = \Sigma_{N+}^+ \oplus \Sigma_{N+}^-$. 
We will consider the Fredholm operator 
\[
(\D_N)_+^+ := \D_N|_{\Sigma_{N+}^+} \colon \Sigma_{N+}^+ \to \Sigma_{N+}^- .
\]
\begin{thm}[{\cite[Theorem 1.5]{Ang93a}}]
\label{thm:classical_Callias}
Under the assumptions given above, we have the equalities
\[
\Index\big(\D-i\pS\big) = \Index \big(\D_N\big)_+^+ = [\Sigma_{N+}] \otimes_{C(N)} [\D_N] = \int_N \hat A(N) \wedge \ch(\Sigma_{N+}) . 
\]
\end{thm}
We note that, while Anghel's theorem and proof focused on the first equality, the index of $\big(\D_N\big)_+^+$ realises the index pairing (Kasparov product) of the $\K$-theory class $[\Sigma_{N+}]$ with the $\K$-homology class $[\D_N]$, and it can be computed as $\int_N \hat A(N) \wedge \ch(\Sigma_{N+})$ by the Atiyah--Singer Index Theorem \cite{AS63}.

\subsection{Spectral flow}

Next we will describe the spectral flow equality \eqref{eq:sf_rel-ind} from the Introduction in detail (see \cref{prop:sf_rel_cpt_family}). 
First, we provide the relevant definitions in the context of Hilbert $C^*$-modules. 

An adjointable operator $F \in \mL_A(E)$ is called \emph{Fredholm} if there exists a \emph{parametrix} $G \in \mL_A(E)$ such that $GF - 1$ and $FG - 1$ are compact operators on $E$. 
If $F$ is Fredholm, we denote by $\Index(F) \in \K_0(A)$ the $\K_0(A)$-valued index of $F$; for the definition of this index, we refer to \cite[\S2.2]{vdD19_Index_DS} and references therein. 

\subsubsection{The relative index of projections}
\label{sec:rel-ind}

Consider two projections $P,Q \in \mL_A(E)$. 
If the difference $P-Q$ is a \emph{compact} operator on $E$, then the operator $Q \colon \Ran(P) \to \Ran(Q)$ is a Fredholm operator with parametrix $P \colon \Ran(Q) \to \Ran(P)$. 
\begin{defn}
\label{defn:rel-ind}
For projections $P,Q \in \mL_A(E)$ with $P-Q \in \mK_A(E)$, we define the \emph{relative index of $(P,Q)$} by 
\[
\relind(P,Q) := \Index \big( Q \colon \Ran(P) \to \Ran(Q) \big) \in \K_0(A) .
\]
\end{defn}

For future reference we record two important properties of the relative index:
\begin{lem}[{\cite[\S3.2]{Wah07}}]
\label{lem:rel-ind_properties}
\begin{itemize}
\item (Additivity.) 
If $P,Q,R \in \mL_A(E)$ are projections with $P-Q$ and $Q-R$ compact, then 
\[
\relind(P,R) = \relind(P,Q) + \relind(Q,R) .
\]
\item (Homotopy invariance.) 
If $\{P_t\}_{t\in[0,1]}$ and $\{Q_t\}_{t\in[0,1]}$ are strongly continuous paths of projections such that $P_t-Q_t$ is compact for each $t\in[0,1]$, then 
\[
\relind(P_0,Q_0) = \relind(P_1,Q_1) .
\]
\end{itemize}
\end{lem}

Using the homotopy invariance of the relative index, we obtain the following: 
\begin{coro}
\label{coro:rel-ind_s-cts_proj}
Let $\{P_t\}_{t\in[0,1]}$ be a strongly continuous family of projections on $E$, such that $P_t-P_0$ is compact for each $t\in[0,1]$. 
Then 
\(
\relind(P_0,P_1) = 0 .
\)
\end{coro}

\subsubsection{The spectral flow}

The notion of spectral flow for a path of self-adjoint operators (typically pa\-ra\-metrised by the unit interval) was first defined by Atiyah and Lusztig, and it appeared in the work of Atiyah, Patodi, and Singer \cite[\S7]{APS76}. 
Heuristically, the spectral flow of a path of self-adjoint Fredholm operators counts the net number of eigenvalues which pass through zero. 
An analytic definition of the spectral flow of a path of self-adjoint Fredholm operators on a Hilbert space was given by Phillips in \cite{Phi96}. 
An axiomatic study of the spectral flow was given by Lesch in \cite{Les05}. 

For regular self-adjoint Fredholm operators on a Hilbert $C^*$-module, a general definition of spectral flow was given by Wahl in \cite[\S3]{Wah07}, which we will largely follow here. 
However, we slightly adapt this definition by allowing the `trivialising operators' (which appear in the definition of the spectral flow) to be possibly unbounded (rather than bounded, as in \cite[\S3]{Wah07}). 
This is made possible by \cref{prop:cpt_diff_cts_funct,coro:cpt_diff_spec-projs} (generalising \cite[Proposition 3.7]{Wah07}). 
For the definition and properties of relatively compact operators, we refer the reader to \S\ref{app:rel-cpt} in the Appendix. 

\begin{defn}[{cf.\ \cite[Definition 3.4]{Wah07}}]
Let $\D$ be a regular self-adjoint operator on $E$. 
A \emph{trivialising operator for $\D$} is a (densely defined) symmetric operator $\B$ on $E$ such that $\B$ is relatively $\D$-compact and $\D+\B$ is invertible. 
\end{defn}

Now let $\D$ be a regular self-adjoint Fredholm operator on $E$, and let $\B_0$ and $\B_1$ be two trivialising operators for $\D$. 
By \cref{coro:cpt_diff_spec-projs}, which generalises \cite[Proposition 3.7]{Wah07}, the difference of spectral projections $P_+(\D+\B_1) - P_+(\D+B_0)$ is compact. Hence we can define
\begin{equation}
\label{eq:ind}
\ind(\D,\B_0,\B_1) := \relind\big( P_+(\D+\B_1) , P_+(\D+B_0) \big) .
\end{equation}

\begin{defn}[{cf.\ \cite[Definition 3.9]{Wah07}}]
Let $X$ be a compact Hausdorff space, and consider a regular operator $\D(\cdot) = \{\D(x)\}_{x\in X}$ on the Hilbert $C(X,A)$-module $C(X,E)$. 
A \emph{trivialising family for $\{\D(x)\}_{x\in X}$} is a family $\{\B(x)\}_{x\in X}$ of operators on $E$ such that $\B(\cdot)$ is a trivialising operator for $\D(\cdot)$. 

We say \emph{there exist locally trivialising families} for $\D(\cdot)$ if for each $x\in X$ there exist a compact neighbourhood $O_x$ of $x$ and a trivialising family for $\{\D(y)\}_{y\in O_x}$.
\end{defn}
We note that the existence of locally trivialising families for $\{\D(x)\}_{x\in X}$ then implies that $\D(\cdot)$ is Fredholm (using compactness of $X$). 

\begin{defn}[{cf.\ \cite[Definition 3.10]{Wah07}}]
\label{defn:spectral_flow}
Let $\D(\cdot) = \{\D(t)\}_{t\in[0,1]}$ be a regular self-adjoint operator on the Hilbert $C([0,1],A)$-module $C([0,1],E)$, for which locally trivialising families exist. 
Let $0 = t_0 < t_1 < \ldots < t_n = 1$ be such that there is a trivialising family $\{\B^i(t)\}_{t\in[t_i,t_{i+1}]}$ of $\{\D(t)\}_{t\in[t_i,t_{i+1}]}$ for each $i=0,\ldots,n-1$. 
Let $\A_0$ and $\A_1$ be trivialising operators of $\D(0)$ and $\D(1)$. 
Then we define 
\begin{align*}
\SF\big(\{\D(t)\}_{t\in[0,1]} ; \A_0,\A_1 \big) 
&:= \ind\big(\D(0),\A_0,\B^0(0)\big) + \sum_{i=1}^{n-1} \ind\big(\D(t_i),\B^{i-1}(t_i),\B^i(t_i)\big) \\
&\qquad+ \ind\big(\D(1),\B^{n-1}(1),\A_1\big) 
\in \K_0(A) ,
\end{align*}
where $\ind$ is defined in \cref{eq:ind}. 
If we assume furthermore that the endpoints $\D(0)$ and $\D(1)$ are invertible, then the \emph{spectral flow of $\{\D(t)\}_{t\in[0,1]}$} is defined by
\begin{multline*}
\SF\big(\{\D(t)\}_{t\in[0,1]} \big) 
:= \SF\big(\{\D(t)\}_{t\in[0,1]} ; 0,0 \big) \\
= \ind\big(\D(0),0,\B^0(0)\big) + \sum_{i=1}^{n-1} \ind\big(\D(t_i),\B^{i-1}(t_i),\B^i(t_i)\big) + \ind\big(\D(1),\B^{n-1}(1),0\big) .
\end{multline*}
\end{defn}

As in \cite{Wah07}, the definition of the spectral flow is independent of the choice of subdivision and the choice of trivialising families $\{\B^i(t)\}_{t\in[t_i,t_{i+1}]}$. 
In particular, using \cite[Lemma 3.5]{Wah07}, we may choose the trivialising families to be \emph{bounded}, and thus we recover the definition of the spectral flow given in \cite[Definition 3.10]{Wah07}.

\subsubsection{Spectral flow for relatively compact perturbations}
\label{sec:sf_rel-cpt}

For our attempt to generalise the Callias Theorem to the case of \emph{infinite-rank} bundles, we take some inspiration from the study of spectral flow. In particular, the following result motivates the idea that a Callias-type theorem should hold for generalised Dirac--Schrödinger operators whenever the `potential' is given by a family of relatively compact perturbations. 
\begin{prop}[{cf.\ \cite[Example in \S3.4]{Wah07}}]
\label{prop:sf_rel_cpt_family}
Let $\mT(\cdot) = \{\mT(t)\}_{t\in[0,1]}$ be a regular self-adjoint operator on the Hilbert $C([0,1],A)$-module $C([0,1],E)$, such that
\begin{itemize}
\item 
the endpoints $\mT(0)$ and $\mT(1)$ are invertible; 
\item 
$\mT(t) \colon \Dom\mT(0) \to E$ depends norm-continuously on $t$; and
\item 
$\mT(t)-\mT(0)$ is relatively $\mT(0)$-compact for each $t\in[0,1]$. 
\end{itemize}
Then the following statements hold:
\begin{enumerate}
\item 
There exists a trivialising family for $\{\mT(t)\}_{t\in[0,1]}$. 
\item 
We have the equality 
\begin{align}
\label{eq:sf_glob}
\SF\big(\{\mT(t)\}_{t\in[0,1]} \big) 
&= \relind\big( P_+(\mT(1)) , P_+(\mT(0)) \big) .
\end{align}
\end{enumerate}
\end{prop}
\begin{proof}
\begin{enumerate}
\item 
We observe that the family of operators $\B(t) := \mT(0)-\mT(t)$ ($t\in[0,1]$) yields a densely defined symmetric operator $\B(\cdot)$ on $C([0,1],E)$, such that $\mT(\cdot)+\B(\cdot)$ is invertible. 
Moreover, $\B(t) \big(\mT(t)\pm i\big)^{-1}$ is compact for each $t\in[0,1]$ (where we use that $\Dom\mT(t)=\Dom\mT(0)$ by \cref{prop:rel_cpt_pert_reg_sa}). 
Since $\big(\mT(t)\pm i\big) \big(\mT(0)\pm i)^{-1}$ is norm-continuous in $t$, also the family of inverses $\big(\mT(0)\pm i\big) \big(\mT(t)\pm i)^{-1}$ is norm-continuous, and therefore $\B(t) \big(\mT(t)\pm i\big)^{-1}$ is norm-continuous in $t$. 
This shows that $\B(\cdot)$ is relatively $\mT(\cdot)$-compact. 
Thus, $\B(\cdot)$ is a trivialising operator for $\mT(\cdot)$. 

\item 
We can insert the trivialising family $\{\B(t)\}_{t\in[0,1]}$ from the first statement into \cref{defn:spectral_flow} to obtain 
\begin{align*}
&\SF\big(\{\mT(t)\}_{t\in[0,1]} \big) 
= \ind\big( \mT(0),0,\B(0) \big) + \ind\big( \mT(1),\B(1),0 \big) \\
&\quad= \relind\big( P_+(\mT(0)+\B(0)) , P_+(\mT(0)) \big) + \relind\big( P_+(\mT(1)) , P_+(\mT(1)+\B(1)) \big) \\
&\quad= \relind\big( P_+(\mT(1)) , P_+(\mT(0)) \big) ,
\end{align*}
where we used that $\B(0) = 0$ and $\mT(1)+\B(1) = \mT(0)$. 
\qedhere
\end{enumerate}
\end{proof}

We remark that for paths of operators on Hilbert \emph{spaces} (rather than Hilbert modules), the identity \eqref{eq:sf_glob} has been shown to hold even under more general continuity assumptions, see \cite[Theorem 3.6]{Les05} and \cite[Proposition 2.5]{Wah08}.

\section{A generalised Callias-type theorem}
\label{sec:Callias}

\subsection{Generalised Dirac--Schrödinger operators}
\label{subsec:DS}

Throughout this section, we will consider the following setting. 
\begin{assumption*}[(A)]
\customlabel{ass:A}{(A)}
Let $A$ be a $\sigma$-unital $C^*$-algebra, and let $E$ be a countably generated Hilbert $A$-module. 
Let $M$ be a connected Riemannian manifold (typically non-compact), and let $\D$ be an essentially self-adjoint elliptic first-order differential operator on a hermitian vector bundle $\bF\to M$.
Let $\{\pS(x)\}_{x\in M}$ be a family of regular self-adjoint operators on $E$ satisfying the following assumptions: 
\begin{description}
\item[(A1)]
\customlabel{ass:A1}{(A1)}
The domain $W := \Dom\pS(x)$ is independent of $x\in M$, and the inclusion $W\into E$ is compact (where $W$ is viewed as a Hilbert $A$-module equipped with the graph norm of $\pS(x_0)$, for some $x_0\in M$). 
\item[(A2)]
\customlabel{ass:A2}{(A2)}
The map $\pS\colon M\to\mL_A(W,E)$ is norm-continuous. 
\item[(A3)]
\customlabel{ass:A3}{(A3)}
There is a compact subset $K\subset M$ such that $\pS(x)$ is uniformly invertible on $M\setminus K$. 
\end{description}
\end{assumption*}

Given the family of operators $\{\pS(x)\}_{x\in M}$ on $E$, we obtain a closed symmetric operator $\pS(\cdot)$ on $C_0(M,E)$, which is defined as the closure of the operator $\big(\pS(\cdot)\psi\big)(x) := \pS(x) \psi(x)$ on the initial dense domain $C_c(M,W)$. 
By \cite[Proposition 3.4]{vdD19_Index_DS}, the operator $\pS(\cdot)$ on the Hilbert $C_0(M,A)$-module $C_0(M,E)$ is regular self-adjoint and Fredholm. 
Consequently, we obtain from \cite[Proposition 2.14]{vdD19_Index_DS} a well-defined $\K$-theory class 
\[
[\pS(\cdot)] \in \KK^1(\C,C_0(M,A)) \simeq \K_1(C_0(M,A)) . 
\]

Furthermore, since $\D$ is an essentially self-adjoint first-order differential operator, and since the ellipticity of $\D$ ensures that $\D$ also has locally compact resolvents \cite[Proposition 10.5.2]{Higson-Roe00}, we know that $(C_0^1(M), L^2(M,\bF), \D)$ is an (odd) spectral triple, which represents a $\K$-homology class 
\[
[\D] \in \KK^1(C_0(M),\C) \simeq \K^1(C_0(M)) \equiv \K_1(M) . 
\]

We consider the balanced tensor product $L^2(M,E\otimes\bF) := C_0(M,E) \otimes_{C_0(M)} L^2(M,\bF)$. 
The operator $\pS(\cdot)\otimes1$ is well-defined on $\Dom\pS(\cdot) \otimes_{C_0(M)} L^2(M,\bF) \subset L^2(M,E\otimes\bF)$, and is denoted simply by $\pS(\cdot)$ as well. By \cite[Proposition 9.10]{Lance95}, $\pS(\cdot)$ is regular self-adjoint on $L^2(M,E\otimes\bF)$. 

The operator $1\otimes\D$ is not well-defined on $L^2(M,E\otimes\bF)$. Instead, using the canonical isomorphism $L^2(M,E\otimes\bF) \simeq E \otimes L^2(M,\bF)$, we consider the operator $1\otimes\D$ on $E \otimes L^2(M,\bF)$ with domain $E\otimes\Dom\D$. 
Alternatively, we can extend the exterior derivative on $C_0^1(M)$ to an operator 
\[
d \colon C_0^1(M,E) \xrightarrow{\simeq} E\otimes C_0^1(M) \xrightarrow{1\otimes d} E\otimes\Gamma_0(T^*M) \xrightarrow{\simeq} \Gamma_0(E\otimes T^*M) .
\]
Denoting by $\sigma$ the principal symbol of $\D$, we can define an operator $1\otimes_d\D$ on the Hilbert space $C_0(M,E) \otimes_{C_0(M)} L^2(M,\bF)$ by setting
$$
(1\otimes_d\D)(\xi\otimes\psi) := \xi\otimes\D\psi + (1\otimes\sigma)(d\xi)\psi .
$$
Under the isomorphism $C_0(M,E) \otimes_{C_0(M)} L^2(M,\bF) \simeq E \otimes L^2(M,\bF)$, the operator $1\otimes\D$ on $E \otimes L^2(M,\bF)$ agrees with $1\otimes_d\D$ on $C_0(M,E) \otimes_{C_0(M)} L^2(M,\bF)$. We will denote this operator on $L^2(M,E\otimes\bF)$ simply as $\D$. The operator $\D$ is regular self-adjoint on $L^2(M,E\otimes\bF)$ (see also \cite[Theorem 5.4]{KL13}). 

\begin{defn}
\label{defn:gen_DS}
Consider $M$, $\D$, and $\pS(\cdot)$ satisfying assumption \ref{ass:A}. 
We define the operator 
\[
\D_\pS := \D - i \pS(\cdot) 
\]
on the initial domain $C_c^1(M,W) \otimes_{C_0^1(M)} \Dom\D$.
Since $\D + i \pS(\cdot) \subset \big( \D - i \pS(\cdot) \big)^*$ is densely defined (on the same domain), $\D-i\pS(\cdot)$ is closable, and (with slight abuse of notation) we denote its closure simply by $\D_\pS$ as well. 

The operator $\D_\pS$ is called a \emph{generalised Dirac--Schrödinger operator} if $\D_\pS$ is regular and Fredholm, and $\D_\pS^* = \D_{-\pS}$. 
In this case, we obtain a well-defined $\K_0(A)$-valued index
\[
\Index \D_\pS \in \K_0(A) .
\] 
\end{defn}
For the definition of this index, we refer to \cite[\S2.2]{vdD19_Index_DS} and references therein. 

We note that, despite our use of the term `Dirac--Schrödinger' operator, we do not assume that the operator $\D$ is of Dirac-type (although a Dirac-type operator is of course the typical example, as described in the Introduction). Furthermore, we note that regularity, the Fredholm property, and the adjoint relation of $\D_\pS$ do not follow automatically from assumption \ref{ass:A}. 

In order to prove the Fredholm property of $\D_\pS$, we consider in addition to assumption \ref{ass:A} also the following assumption:
\begin{assumption*}[(B)]
\customlabel{ass:B}{(B)}
We assume the following conditions are satisfied:
\begin{description}
\item[(B1)]
\customlabel{ass:B1}{(B1)}
the map $\pS \colon M \to \mL_A(W,E)$ is weakly differentiable (i.e., for each $\psi\in W$ and $\eta\in E$, the map $x \mapsto \la\pS(x)\psi|\eta\ra$ is differentiable), and the weak derivative $d\pS(x) \colon W \to E\otimes T_x^*(M)$ is bounded for all $x\in M$. 
\item[(B2)]
\customlabel{ass:B2}{(B2)} 
the operator $\big[\D,\pS(\cdot)\big] \big(\pS(\cdot)\pm i\big)^{-1}$ is well-defined and bounded (in the sense of \cite[Assumption 7.1]{KL12} and \cite[Definition 5.5]{vdD19_Index_DS}): 
there exists a core $\E\subset\Dom\D$ for $\D$ such that for all $\xi\in\E$ and for all $\mu\in(0,\infty)$ we have the inclusions 
\[
\big(\pS(\cdot)\pm i\mu\big)^{-1} \xi \in \Dom\pS(\cdot) \cap \Dom\D 
\quad \text{and} \quad 
\D \big(\pS(\cdot)\pm i\mu\big)^{-1} \xi \in \Dom\pS(\cdot) ,
\]
and the map $\big[\D,\pS(\cdot)\big] \big(\pS(\cdot)\pm i\mu\big)^{-1} \colon \E \to L^2(M,E\otimes\bF)$ extends to a bounded operator for all $\mu\in(0,\infty)$. 
\end{description}
\end{assumption*}

\begin{remark}
\label{remark:B}
\begin{enumerate}
\item 
Assumption \ref{ass:B} requires the potential $\pS(\cdot)$ to be differentiable (in a suitable sense). Alternatively, it is also possible to deal with continuous potentials, as is done in \cite{vdD19_Index_DS}. 
\item 
As described in \cite[Remark 8.4]{KL13}, assumption \ref{ass:B1} already implies assumption \ref{ass:A2}. 
\item \label{item:bdd_comm}
If, in addition to \ref{ass:B1}, we assume that $\D$ has bounded propagation speed, that $\Dom\pS(\cdot) \subset C_0(M,W)$ (i.e., that there exists $C>0$ such that for all $x\in M$ we have $\|\cdot\|_W \leq C \|\cdot\|_{\pS(x)}$), and that the weak derivative $d\pS(\cdot)$ is \emph{uniformly} bounded, then the boundedness of $\big[\D,\pS(\cdot)\big] \big(\pS(\cdot)\pm i\big)^{-1}$ in \ref{ass:B2} already follows. 
Indeed, as in \cite[Lemma 8.5 \& Theorem 8.6]{KL13}, we can then write
\begin{multline*}
\big[\D,\pS(\cdot)\big] \big(\pS(\cdot)\pm i\big)^{-1} 
= \sigma_\D \circ d\pS(\cdot) \circ \big(\pS(\cdot)\pm i\big)^{-1} 
\colon \\
L^2(M,E\otimes\bF) 
\xrightarrow{(\pS(\cdot)\pm i)^{-1}} \Dom\pS(\cdot) \otimes_{C_0(M)} L^2(M,\bF) 
\into L^2(M,W\otimes\bF) \\
\xrightarrow{d\pS(\cdot)} L^2(M,E\otimes T^*M\otimes\bF) 
\xrightarrow{\sigma_\D} L^2(M,E\otimes\bF) 
\end{multline*}
as a composition of bounded operators. 
\end{enumerate}
\end{remark}

Thanks to assumption \ref{ass:B}, we have: 
\begin{prop}[{\cite[Theorem 7.10]{KL12}}]
\label{prop:reg-sa}
The operators $\D_{\pm\pS}$ are regular on the domain $\Dom\D_\pS$ and satisfy $\D_{\pm\pS}^* = \D_{\mp\pS}$. 
\end{prop}

The following theorem will be proven as the first statement of \cref{thm:Fredholm} below. It states that the operator $\D_\pS$ is Fredholm, provided that (if necessary) the potential $\pS(\cdot)$ is rescaled by a sufficiently large $\lambda>0$. 
\begin{thm}
\label{thm:Fred}
There exists $\lambda_0>0$ such that for any $\lambda\geq\lambda_0$ the operator $\D_{\lambda\pS}$ is Fredholm and thus a generalised Dirac--Schrödinger operator. 
\end{thm}

Our next theorem then describes the Fredholm index of a Dirac--Schrödinger operator in terms of the \emph{index pairing} between the $\K$-theory class of the potential $\pS(\cdot)$ and the $\K$-homology class of the elliptic operator $\D$. Results of this form were previously given by Bunke \cite{Bun95} (see also \cite{Kuc01}) in the classical case, and in \cite{KL13,vdD19_Index_DS} for `generalised' Dirac--Schrödinger operators. 
\begin{restatable}{thm}{thmKaspprodindex}
\label{thm:Kasp_prod_index}
Let $M$ be a connected Riemannian manifold, and let $\{\pS(x)\}_{x\in M}$ and $\D$ satisfy assumptions \ref{ass:A} and \ref{ass:B}. 
Then there exists $\lambda_0>0$ such that for any $\lambda\geq\lambda_0$ 
the $\K_0(A)$-valued index of $\D_{\lambda\pS}$ equals the pairing of $[\pS(\cdot)] \in \K_1(C_0(M,A))$ with $[\D] \in \K^1(C_0(M))$. 
\end{restatable}
The proof is given in \S\ref{sec:Kasp_prod}. It relies on identifying the classes as elements in Kasparov's $\KK$-theory via the isomorphisms $\K_1(C_0(M,A)) \simeq \KK^1(\C,C_0(M,A))$, $\K^1(C_0(M)) \simeq \KK^1(C_0(M),\C)$, and $\K_0(A) \simeq \KK^0(\C,A)$, and then computing the index pairing using the description of the unbounded Kasparov product given in \cite{KL13}.

\subsection{Generalised Callias-type operators}
\label{subsec:Callias}

Let $M$, $\D$ and $\pS(\cdot)$ satisfy assumptions \ref{ass:A} and \ref{ass:B} such that $\D_{\lambda\pS}$ is Fredholm (and hence a generalised Dirac--Schrödinger operator) for $\lambda \geq \lambda_0 > 0$. 
In the remainder of this section, we furthermore assume: 

\begin{assumption*}[(C)]
\customlabel{ass:C}{(C)}
Without loss of generality, assume that the compact subset $K$ from assumption \ref{ass:A3} has a smooth compact boundary $N$. 
We assume furthermore that the following conditions are satisfied:
\begin{description}
\item[(C1)]
\customlabel{ass:C1}{(C1)}
The operator $\D$ is of `product form' near $N$ in the following sense. 
There exists a collar neighbourhood $C \simeq (-2\varepsilon,2\varepsilon) \times N$ of $N$ (with $(-2\varepsilon,0) \times N$ in the interior of $K$), where we can identify $\bF|_C$ with the pullback of $\bF_N := \bF|_N \to N$ to $C \simeq (-2\varepsilon,2\varepsilon) \times N$, so that $\Gamma^\infty(\bF|_C) \simeq C^\infty\big( (-2\varepsilon,2\varepsilon) \big) \otimes \Gamma^\infty(\bF_N)$. 
On this collar neighbourhood 
we have $\D|_C \simeq -i\partial_r \otimes \Gamma_N + 1 \otimes \D_N$, where $\D_N$ is an essentially self-adjoint elliptic first-order differential operator on $\bF_N\to N$, and where $\Gamma_N \in \Gamma^\infty(\End\bF_N)$ is a self-adjoint unitary satisfying $\Gamma_N \D_N = - \D_N \Gamma_N$. 
\item[(C2)]
\customlabel{ass:C2}{(C2)}
For any $x,y\in K$, $\pS(x)-\pS(y)$ is relatively $\pS(x)$-compact. 
\end{description}
Moreover, we fix an (arbitrary) invertible regular self-adjoint operator $\pT$ on $E$ with domain $\Dom\pT = W$, such that $\pS(x)-\pT$ is relatively $\pT$-compact for some (and hence, by \ref{ass:C2}, for every) $x\in K$. 
\end{assumption*}

\begin{remark}
\begin{enumerate}
\item 
For the definition and properties of relatively compact operators, we refer the reader to \S\ref{app:rel-cpt} in the Appendix. 
\item 
The product form of $\D$ in assumption \ref{ass:C1} is typical of Dirac operators corresponding to a product metric on the collar neighbourhood $C$ of $N$, where $\Gamma_N$ is given by Clifford multiplication with the unit normal vector $\partial_r$ to $N$ (actually, one might often write $\D_C' = -i (1\otimes\Gamma_N) (\partial_r \otimes 1 + 1 \otimes \D_N)$, but these two product forms are in fact unitarily equivalent). 
However, in this paper, we do not insist that $\D$ is of Dirac-type. 
One can view assumption \ref{ass:C1} as requiring precisely those properties of Dirac operators which we need below (in particular, to prove \cref{lem:D_product}). 

\item 
We remind the reader that assumption \ref{ass:C2} is motivated by the spectral flow result from \cref{prop:sf_rel_cpt_family}. 

\item 
We note that, for the operator $\pT$, we can for instance choose $\pT = \pS(x_0)$ for some $x_0\in K$, but it can be useful to allow for arbitrary relatively compact perturbations. 
\end{enumerate}
\end{remark}

\begin{defn}
If assumptions \ref{ass:A}, \ref{ass:B}, and \ref{ass:C} are satisfied, then the generalised Dirac--Schrödinger operator $\D_{\lambda\pS}$ is called a \emph{generalised Callias-type operator}. 

\noindent 
(We always implicitly assume that $\lambda\geq\lambda_0>0$ such that $\D_{\lambda\pS}$ is Fredholm.)
\end{defn}

We consider the invertible regular self-adjoint operator $\pT(\cdot)$ on $C(N,E)$ corresponding to the constant family $\pT(y) := \pT$ (for $y\in N$). 
The restriction of the potential $\pS(\cdot)$ to the hypersurface $N$ also yields an invertible regular self-adjoint operator $\pS_N(\cdot) = \{\pS(y)\}_{y\in N}$ on $C(N,E)$. 
We recall that $\pS(y) - \pT$ is relatively $\pT$-compact for each $y\in N$. 
Furthermore, $\pS(y) \big( \pT\pm i \big)^{-1}$ depends norm-continuously on $y$ by assumption \ref{ass:A2}. 
Hence $\pS_N(\cdot) - \pT(\cdot)$ is relatively $\pT(\cdot)$-compact. 
We then know from \cref{coro:cpt_diff_spec-projs} that the difference of positive spectral projections $P_+(\pS_N(\cdot)) - P_+(\pT(\cdot))$ is compact, so that the relative index $\relind\big(P_+(\pS_N(\cdot)),P_+(\pT(\cdot))\big)$ is well-defined in $\K_1(C(N,A))$ (see \cref{defn:rel-ind}). 
We are now ready to state our generalisation of the Callias Theorem. 

\begin{restatable}[Generalised Callias Theorem]{thm}{thmGeneralCallias}
\label{thm:general_Callias}%
Let $\D_{\lambda\pS}$ be a generalised Callias-type operator. 
Then we have the equality
\begin{align*}
\Index\big( \D_{\lambda\pS} \big)
= \relind\big(P_+(\pS_N(\cdot)),P_+(\pT(\cdot))\big) \otimes_{C(N)} [\D_N] \in \K_0(A) ,
\end{align*}
where $\otimes_{C(N)}$ denotes the pairing $\K_1(C(N,A)) \times \K^1(C(N)) \to \K_0(A)$. 
\end{restatable}
\begin{remark}
Although the relative index depends explicitly on the choice of $\pT$, the theorem in particular shows that the pairing $\relind\big(P_+(\pS_N(\cdot)),P_+(\pT(\cdot))\big) \otimes_{C(N)} [\D_N]$ on the right-hand-side is in fact independent of this choice. 
This independence of $\pT$ can be understood as a consequence of the cobordism invariance of the index (since $N$ is the boundary of $K$, the index of $\D_N$ vanishes). In fact, one can also turn this around, and prove the cobordism invariance of the index as a consequence of the Callias Theorem (by considering the trivial rank-one bundle $M\times\C$ with the potential $\pS(\cdot)=1$, and the operator $\pT=-1$ on $E=A=\C$). 
\end{remark}

We observe next that our assumption \ref{ass:C2} ensures that the class $[\pS(\cdot)]$ of the potential depends only on the hypersurface $N$. This is the crucial observation which enables one to obtain the index of the Callias-type operator from a computation on the hypersurface $N$, as in \cref{eq:index_pairings} in the Introduction. 
Consider the open subset $U := K\cup C \subset M$ with compact closure $\bar U$ and boundary $\partial U\simeq N$. 
We have the short exact sequence 
\begin{align}
\label{eq:ses_U}
0 &\rightarrow C_0(U,A) \xrightarrow{j} C(\bar U,A) \to C(N,A) \to 0 
\end{align}
and the corresponding cyclic six-term exact sequences in $\K$-theory and $\K$-homology. 

\begin{prop}
\label{prop:potential_boundary}
The $\K$-theory class $[\pS(\cdot)] \in \K_1(C_0(M,A))$ is uniquely determined by an element $\Sigma_N \in \K_0(C(N,A))$. 
More explicitly, we have 
\[
[\pS(\cdot)] = {\iota_U}_* \circ \partial(\Sigma_N) ,
\]
where $\partial \colon \K_0(C(N,A)) \to \K_1(C_0(U,A))$ denotes the exponential map in the cyclic six-term exact sequence in $\K$-theory corresponding to the short exact sequence \eqref{eq:ses_U}, and where ${\iota_U}_* \colon \K_1(C_0(U,A)) \to \K_1(C_0(M,A))$ is induced by the inclusion $\iota_U \colon C_0(U,A) \into C_0(M,A)$. 
\end{prop}
\begin{proof}
The invertibility of the potential $\pS(\cdot)$ outside of the compact subset $K$ ensures that the class $[\pS(\cdot)]$ depends only on the restriction of $\pS(\cdot)$ to $U$. 
Indeed, we have from \cite[Lemma 3.8]{vdD19_Index_DS} the equality $[\pS(\cdot)] = {\iota_U}_*\big([\pS(\cdot)|_U]\big)$. 

We may assume, without loss of generality, that assumption \ref{ass:C2} holds for all $x\in\bar U$ (see \cref{lem:collar_potential} below for an explicit computation). 
Using compactness of $\bar U$, it follows that the operator $\pS(\cdot)|_{\bar U}$ is a relatively compact perturbation of the invertible operator $\pT(\cdot)_{\bar U} = \{\pT\}_{x\in \bar U}$, and therefore $j_*\big(\big[\pS(\cdot)|_{U}\big]\big) = 0 \in \K_1(C(\bar U,A))$. 
From the cyclic six-term exact sequence in $\K$-theory we conclude that $\big[\pS(\cdot)|_{U}\big]$ lies in the image of the exponential map $\partial \colon \K_0(C(N,A)) \to \K_1(C_0(U,A))$, 
so there exists a class $\Sigma_N \in \K_0(C(N,A))$ such that $\partial(\Sigma_N) = [\pS(\cdot)|_U]$. 
\end{proof}

The above proposition ensures that we can apply \cref{eq:index_pairings}, and combined with \cref{thm:Kasp_prod_index} we obtain the equality 
\[
\Index\big( \D_{\lambda\pS} \big)
= \Sigma_N \otimes_{C(N)} [\D_N] .
\]
Thus, in order to prove \cref{thm:general_Callias}, it remains to explicitly identify the $\K$-theory class $\Sigma_N \in \K_0(C(N,A))$ as the relative index of the positive spectral projections $P_+(\pS_N(\cdot))$ and $P_+(\pT(\cdot))$. 
We will obtain this identification in \cref{sec:proof} by first reducing the general statement to the special case of a cylindrical manifold $\R\times N$ (see \cref{thm:red_cyl}). 
The main advantage of considering the cylindrical manifold is, roughly speaking, that we can then invert the boundary map in order to explicitly compute a solution $\Sigma_N$ of the equation $[\pS(\cdot)|_U] = \partial(\Sigma_N)$.

\subsection{Special cases}
\label{subsec:special_cases}

In this subsection we reconsider the two well-known special cases of our generalised Callias Theorem, described in \cref{sec:background}. 

First, in the special case when $E$ is a finite-dimensional Hilbert space, we recover the classical Callias \cref{thm:classical_Callias} (though only for globally trivial bundles). In fact, we find that the statement of the classical Callias Theorem continues to hold if $E$ is a finitely generated projective module over a \emph{unital} $C^*$-algebra $A$.

\begin{coro}
\label{coro:classical_Callias}
Let $\D_{\lambda\pS}$ be a generalised Callias-type operator. 
Suppose furthermore that $A$ is unital, and that $E$ is finitely generated and projective over $A$. 
Then 
\[
\Index\big( \D_{\lambda\pS} \big)
= \big[ \Ran P_+(\pS_N(\cdot)) \big] \otimes_{C(N)} [\D_N] \in \K_0(A) .
\]
\end{coro}
\begin{proof}
The assumptions on $A$ and $E$ ensure that all operators on $E$ are compact. 
In particular, the operator $\pT:=-1$ is a relatively compact perturbation of each $\pS_N(y)$. 
With $P_+(\pT(\cdot)) = 0$ we therefore obtain
\[
\relind\big(P_+(\pS_N(\cdot)),0\big) = \Index\big( 0 \colon \Ran P_+(\pS_N(\cdot)) \to \{0\} \big) = \big[ \Ran P_+(\pS_N(\cdot)) \big] . 
\]
Thus from \cref{thm:general_Callias} we find that 
\begin{align*}
\Index\big( \D_{\lambda\pS} \big)
&= \relind\big(P_+(\pS_N(\cdot)),0\big) \otimes_{C(N)} [\D_N] 
= \big[ \Ran P_+(\pS_N(\cdot)) \big] \otimes_{C(N)} [\D_N] .
\qedhere
\end{align*}
\end{proof}

Second, in the special case where $M=\R$, we recover the equality between the spectral flow and the relative index of spectral projections of the end-points from \cref{prop:sf_rel_cpt_family}. 
\begin{coro}
\label{coro:Callias_sf}
Consider the operator $\D = -i\partial_t$ on the manifold $M=\R$ and a potential $\pS(\cdot) = \big\{ \pS(t) \big\}_{t\in\R}$ satisfying assumptions \ref{ass:A}, \ref{ass:B}, and \ref{ass:C}. 
Suppose for simplicity that the compact subset from assumption \ref{ass:A3} is given by the unit interval $K = [0,1]$. 
Then we have the equality 
\begin{align*}
\SF\big(\{\pS(t)\}_{t\in[0,1]}\big)
&= \relind\big(P_+(\pS(1)),P_+(\pS(0))\big) . 
\end{align*}
\end{coro}
\begin{proof}
From \cref{thm:general_Callias} we obtain 
\[
\Index\big( \D - i \lambda \pS(\cdot) \big)
= \relind\big(P_+(\pS_N(\cdot)),P_+(\pT(\cdot))\big) \otimes_{C(N)} [\D_N] ,
\]
where the `hypersurface' $N = \{0,1\}$ consists of the endpoints of the unit interval, and $\pT$ is any relatively compact perturbation of $\pS(0)$. 
We will examine both the left-hand-side and the right-hand-side of the above equation. 

First, the left-hand-side is given by 
\[
\Index\big( {-i\partial_t} - i \lambda \pS(\cdot) \big)
= [\pS(\cdot)] \otimes_{C_0(\R)} [-i\partial_t] 
= \SF\big(\{\pS(t)\}_{t\in[0,1]}\big) , 
\]
where the first equality is obtained from \cref{thm:Kasp_prod_index}, and the second from \cite[Proposition 2.21]{vdD19_Index_DS} (using that trivialising families exist by \cref{prop:sf_rel_cpt_family}.(1)). 

For the right-hand-side, we examine the product form of $-i\partial_t$ near $N=\{0,1\}$. 
The operator $\D_N$ is just the zero operator on $\bF_N = \bF_{\{0\}} \oplus \bF_{\{1\}} \simeq \C\oplus\C$. 
We note that the coordinate $r$ increases in the outward direction, so we have $t=-r$ near $0$ and $t=1+r$ near $1$. 
Thus, on a collar neighbourhood of $N$ we can write $-i\partial_t \simeq i\partial_r \oplus (-i\partial_r) = -i\partial_r \otimes \Gamma_N$, where the operator $\Gamma_N$ is given by $(-1) \oplus 1$ on $\bF_{\{0\}} \oplus \bF_{\{1\}}$. 
Thus $\bF_{\{0\}} = \bF^-_{\{0\}} = \C$ and $\bF_{\{1\}} = \bF^+_{\{1\}} = \C$, and we can identify $[\D_N]\in\KK^0(\C^2,\C) \simeq \K^0(\C^2)$ with $(-1)\oplus1 \in \Z\oplus\Z$. 
Then the Kasparov product over $C(N)=\C^2$ can be calculated as follows:
\begin{align*}
\relind&\big(P_+(\pS_N(\cdot)),P_+(\pT(\cdot))\big) \otimes_{\C^2} [\D_N] \\
&= \relind\big(P_+(\pS(0)),P_+(\pT)\big) \otimes (-1) + \relind\big(P_+(\pS(1)),P_+(\pT)\big) \otimes 1 \\
&= \relind\big(P_+(\pT),P_+(\pS(0))\big) + \relind\big(P_+(\pS(1)),P_+(\pT)\big) \\
&= \relind\big(P_+(\pS(1)),P_+(\pS(0))\big) . 
\qedhere
\end{align*}
\end{proof}

\section{Generalised Dirac--Schrödinger operators}
\label{sec:DS}

Consider $M$, $\D$, and $\pS(\cdot)$ satisfying assumption \ref{ass:A}. 
We have defined in \cref{defn:gen_DS} the operator \( \D_\pS := \D - i \pS(\cdot) \) on the initial domain $C_c^1(M,W) \otimes_{C_0^1(M)} \Dom\D$.
We now also define the operators
\begin{gather*}
\til\D := \mattwo{0}{\D}{\D}{0} , \qquad 
\til\pS(\cdot) := \mattwo{0}{+i\pS(\cdot)}{-i\pS(\cdot)}{0} , \\
\til\D_\pS := \til\D + \til\pS(\cdot) = \mattwo{0}{\D+i\pS(\cdot)}{\D-i\pS(\cdot)}{0} ,
\end{gather*}
on the initial domain $\big(C_c^1(M,W) \otimes_{C_0^1(M)} \Dom\D\big)^{\oplus2}$. 
The operator $\til\D_\pS$ is odd with respect to the $\Z_2$-grading $\Gamma := \mattwo{1}{0}{0}{-1}$. 

We recall that $\D_\pS$ is called a \emph{generalised Dirac--Schrödinger operator} if $\til\D_\pS$ is regular, self-adjoint, and Fredholm. 
In this case, the operator $\til\D_\pS$ yields a class 
\[
[\til\D_\pS] \in \KK^0(\C,A) ,
\] 
corresponding to the $\K_0(A)$-valued index of $\D_\pS$ under the isomorphism $\KK^0(\C,A) \simeq \K_0(A)$. 
For the construction of this class in Kasparov's $\KK$-theory and its relation to the Fredholm index, we refer to \cite[\S2.2]{vdD19_Index_DS}.

\subsection{Relative index theorem}
\label{sec:rel_index_thm}

An important tool for our index computations is the relative index theorem \cite[Theorem 4.7]{vdD19_Index_DS}, which is an adaptation of a theorem by Bunke \cite[Theorem 1.14]{Bun95}. 
Here we shall adapt \cite[Theorem 4.7]{vdD19_Index_DS} in order to allow for more general situations (in particular, we avoid the assumption (A4) from \cite[\S3.2]{vdD19_Index_DS}). 

We consider the following setting. 
For $j=1,2$, let $\bF^j\to M^j$, $\D^j$, and $\pS^j(\cdot)$ be as in assumption \ref{ass:A}, 
and assume that the operators $\{\pS^j(x)\}_{x\in M^j}$ act on the same Hilbert $A$-module $E$. 
Suppose we have partitions $M^j = \bar U^j \cup_{N^j} \bar V^j$, where $N^j$ are smooth compact hypersurfaces. 
Let $C^j$ be open tubular neighbourhoods of $N^j$, and assume that there exists an isometry $\phi\colon C^1\to C^2$ (with $\phi(N^1)=N^2$) covered by an isomorphism $\Phi\colon\bF^1|_{C^1} \to \bF^2|_{C^2}$, such that $\D^1|_{C^1} \Phi^* = \Phi^* \D^2|_{C^2}$ and $\pS^2(\phi(x)) = \pS^1(x)$ for all $x\in C^1$. 

We will identify $C^1$ with $C^2$ (as well as $N^1$ with $N^2$) via $\phi$, and we simply write $C$ (and $N$). Define two new Riemannian manifolds 
\begin{align*}
M^3 &:= \bar U^1 \cup_N \bar V^2 , & 
M^4 &:= \bar U^2 \cup_N \bar V^1 .
\end{align*}
Moreover, we glue the bundles using $\Phi$ to obtain hermitian vector bundles $\bF^3\to M^3$ and $\bF^4\to M^4$. For $j=3,4$, we then obtain corresponding operators $\D^j$ and $\pS^j(\cdot)$ satisfying assumption \ref{ass:A}. 

\begin{thm}[Relative index theorem]
\label{thm:rel_index}
Assume that $\til\D_\pS^j$ (for $j=1,2$) are regular self-adjoint Fredholm operators with locally compact resolvents. 
Then $\til\D_\pS^3$ and $\til\D_\pS^4$ are also regular self-adjoint Fredholm operators with locally compact resolvents. 
Moreover, we have in $\K_0(A)$ the equality 
\[
\Index \big( \D^1 - i \pS^1(\cdot) \big) + \Index \big( \D^2 - i \pS^2(\cdot) \big) = \Index \big( \D^3 - i \pS^3(\cdot) \big) + \Index \big( \D^4 - i \pS^4(\cdot) \big) .
\]
\end{thm}
\begin{proof}
First, we need to check that $\til\D_\pS^3$ and $\til\D_\pS^4$ are also regular self-adjoint and Fredholm. We give the proof only for $\til\D_\pS^3$. 
We choose smooth functions $\chi_1$ and $\chi_2$ such that 
\begin{align*}
\supp \chi_1 &\subset U^1 \cup C , & 
\supp \chi_2 &\subset V^2 \cup C , &
\chi_1^2 + \chi_2^2 &= 1 .
\end{align*}
For $\lambda>0$ we define 
\[
R_\pm(\lambda) := \chi_1 \big( \til\D_\pS^1 \pm i\lambda \big)^{-1} \chi_1 + \chi_2 \big( \til\D_\pS^2 \pm i\lambda \big)^{-1} \chi_2 .
\]
Then 
\[
\big( \til\D_\pS^3 \pm i\lambda \big) R_\pm(\lambda) = 1 + [\til\D^1,\chi_1] \big(\til\D_\pS^1\pm i\lambda\big)^{-1} \chi_1 + [\til\D^2,\chi_2] \big(\til\D_\pS^2\pm i\lambda\big)^{-1} \chi_2 =: 1 + \mK_\pm(\lambda) .
\]
We can pick $\lambda$ sufficiently large, such that the norm of $\mK_\pm(\lambda)$ is less than one. Then $1+\mK_\pm(\lambda)$ is invertible, and $R_\pm(\lambda)(1+\mK_\pm(\lambda))^{-1}$ is a right inverse of $\til\D_\pS^3\pm i\lambda$. Similarly, we can also obtain a left inverse, which proves that $\til\D_\pS^3$ is regular self-adjoint. 
Moreover, since $R_\pm(\lambda)$ is locally compact, we see that $\til\D_\pS^3$ has locally compact resolvents. 

Next, given parametrices $Q_1$ and $Q_2$ for $\til\D_\pS^1$ and $\til\D_\pS^2$, respectively, we define
\[
Q_3 := \chi_1 Q_1 \chi_1 + \chi_2 Q_2 \chi_2 .
\]
Then 
\[
\til\D_\pS^3 Q_3 - 1 = \chi_1 \big( \til\D_\pS^1 Q_1 - 1 \big) \chi_1 + [\til\D^1,\chi_1] Q_1 \chi_1 + \chi_2 \big( \til\D_\pS^2 Q_2 - 1 \big) \chi_2 + [\til\D^2,\chi_2] Q_2 \chi_2 .
\]
The terms $\chi_j \big( \til\D_\pS^j Q_j - 1 \big) \chi_j$ are compact because $Q_j$ are parametrices. 
Furthermore, the terms $[\til\D^j,\chi_j] Q_j \chi_j$ are compact because $[\til\D^j,\chi_j]$ are compactly supported and $\til\D_\pS^j$ have locally compact resolvents. 
Hence $Q_3$ is a right parametrix for $\til\D_\pS^3$. A similar calculation shows that $Q_3$ is also a left parametrix, and therefore $\til\D_\pS^3$ is Fredholm. 
Similarly, the operator $\til\D_\pS^4$ is also regular self-adjoint and Fredholm. 
The proof of the desired equality $\Index \big( \D^1 - i \pS^1(\cdot) \big) + \Index \big( \D^2 - i \pS^2(\cdot) \big) = \Index \big( \D^3 - i \pS^3(\cdot) \big) + \Index \big( \D^4 - i \pS^4(\cdot) \big) \in \K_0(A)$
is then exactly as in \cite[Theorem 4.7]{vdD19_Index_DS}.
\end{proof}

\subsection{The Fredholm index}
\label{sec:Fred_index}

From here on, we consider $M$, $\D$, and $\pS(\cdot)$ satisfying assumptions \ref{ass:A} and \ref{ass:B}. 
Our aim in this subsection is to prove \cref{thm:Fred} (see \cref{thm:Fredholm} below). 
We first observe that, thanks to assumption \ref{ass:B}, the operator $\til\D_\pS$ has locally compact resolvents. 
\begin{prop}[{\cite[Theorem 6.7]{KL13}}]
\label{prop:cpt_res}
The operator $\phi(\til\D_\pS\pm i)^{-1}$ on $L^2(M,E\otimes\bF)^{\oplus2}$ is compact for any $\phi\in C_0(M)$. 
Moreover, if $(\pS(\cdot)\pm i)^{-1}$ is compact on $C_0(M,E)$, then $(\til\D_\pS\pm i)^{-1}$ is also compact. 
\end{prop}

In order to prove the Fredholm property of $\til\D_{\pS}$, we need to rescale the potential $\pS(\cdot)$ by a sufficiently large $\lambda>0$. 
First, we need the following pointwise estimate. 

\begin{lem}
\label{lem:anti-comm_estimate}
There exist $\lambda_0>0$ and $\epsilon>0$ such that for any $\lambda\geq\lambda_0$ there exists a compactly supported smooth function $f \in C_c^\infty(M)$ such that for all $x\in M$ and $\psi(x) \in (W\otimes\bF)^{\oplus2}$ we have the inequality 
\begin{align*}
\big\lla \lambda\til\pS(x) \psi(x) \big\rra + \big\la \{\til\D,\lambda\til\pS(\cdot)\}(x) \psi(x) \bigmvert \psi(x) \big\ra + \big\lla f(x) \psi(x) \big\rra 
&\geq \epsilon \big\lla \psi(x) \big\rra .
\end{align*}
\end{lem}
\begin{proof}
We roughly follow the proof of \cite[Lemma 5.8]{vdD19_Index_DS}, but with somewhat different estimates. 

First, since $\lambda\pS$ also satisfies assumption \ref{ass:B}, we know from Propositions \ref{prop:reg-sa} and \ref{prop:cpt_res} that $\til\D_{\lambda\pS}$ is regular self-adjoint and has locally compact resolvents. For any $\alpha\in(0,\infty)$, $x\in M$, and $\psi(x) \in (W\otimes\bF)^{\oplus2}$, we have (using the same arguments as in the proof of \cite[Lemma 7.5]{KL12})
\begin{align*}
\pm2 \big\la \{\til\D,\til\pS(\cdot)\}(x) \psi(x) \bigmvert \psi(x) \big\ra 
&\leq \alpha^2 \big\lla \{\til\D,\til\pS(\cdot)\}(x) \psi(x) \big\rra + \alpha^{-2} \big\lla \psi(x) \big\rra ,
\end{align*}
where $\{\cdot,\cdot\}$ denotes the anti-commutator. 
Using that $\delta_x := \big\| \big[\D,\pS(\cdot)\big](x) \big(\pS(x)\pm i\big)^{-1} \big\|$ is bounded, we obtain 
\begin{align}
\label{eq:comm_est}
\pm2 \big\la \{\til\D,\til\pS(\cdot)\}(x) \psi(x) \bigmvert \psi(x) \big\ra 
&\leq \alpha^2 \delta_x^2 \big\lla (\til\pS(x)\pm i) \psi(x) \big\rra + \alpha^{-2} \big\lla \psi(x) \big\rra .
\end{align}
We distinguish between the cases $x\in M\setminus K$ and $x\in K$:
\begin{description}
\item[$\boldsymbol{x\in M\setminus K}$:]
Let $c \equiv c_{M\setminus K} := \inf_{x\in M\setminus K} \|\pS(x)^{-1}\|^{-1}$. 
Then combining \eqref{eq:comm_est} with the norm inequality $\|(\til\pS(x)\pm i) \til\pS(x)^{-1}\| \leq 1+c^{-1}$ we obtain 
\begin{align*}
\pm2 \big\la \{\til\D,\til\pS(\cdot)\}(x) \psi(x) \bigmvert \psi(x) \big\ra 
&\leq \alpha^2 \delta_x^2 (1+c^{-1})^2 \big\lla \til\pS(x) \psi(x) \big\rra + \alpha^{-2} \big\lla \psi(x) \big\rra .
\end{align*}
Now setting $\alpha = \lambda^{1/2} \delta_x^{-1} (1+c^{-1})^{-1}$ yields 
\begin{align*}
\pm2 \big\la \{\til\D,\til\pS(\cdot)\}(x) \psi(x) \bigmvert \psi(x) \big\ra 
&\leq \lambda \big\lla \til\pS(x) \psi(x) \big\rra + \lambda^{-1} \delta_x^2 (1+c^{-1})^2 \big\lla \psi(x) \big\rra 
\end{align*}
and in particular 
\begin{align*}
\big\la \{\til\D,\til\pS(\cdot)\}(x) \psi(x) \bigmvert \psi(x) \big\ra 
&\geq -\frac12 \lambda \big\lla \til\pS(x) \psi(x) \big\rra - \frac12 \lambda^{-1} \delta_x^2 (1+c^{-1})^2 \big\lla \psi(x) \big\rra .
\end{align*}
Thus we have 
\begin{align*}
&\big\lla \lambda\til\pS(x) \psi(x) \big\rra + \big\la \{\til\D,\lambda\til\pS(\cdot)\}(x) \psi(x) \bigmvert \psi(x) \big\ra \\
&\qquad\geq \frac12 \big\lla \lambda\til\pS(x) \psi(x) \big\rra - \frac12 \delta_x^2 (1+c^{-1})^2 \big\lla \psi(x) \big\rra \\
&\qquad\geq \frac12 \big( \lambda^2 c^2 - \delta_x^2 (1+c^{-1})^2 \big) \big\lla \psi(x) \big\rra .
\end{align*}
Now set $\delta_{M\setminus K} := \sup_{x\in M\setminus K} \delta_x = \sup_{x\in M\setminus K} \big\| \big[\D,\pS(\cdot)\big](x) \big(\pS(x)\pm i\big)^{-1} \big\|$, and 
pick $\lambda_0 > 0$ large enough such that $\epsilon := \frac12 \big( \lambda_0^2 c^2 - \delta_{M\setminus K}^2 (1+c^{-1})^2 \big) > 0$. 
Then we have shown that for all $x\in M\setminus K$ and all $\lambda\geq\lambda_0$ we have 
\begin{align}
\label{eq:comm_est_M-K}
\big\lla \lambda\til\pS(x) \psi(x) \big\rra + \big\la \{\til\D,\lambda\til\pS(\cdot)\}(x) \psi(x) \bigmvert \psi(x) \big\ra 
&\geq \epsilon \big\lla \psi(x) \big\rra .
\end{align}

\item[$\boldsymbol{x\in K}$:]
Set $\delta_K := \sup_{x\in K} \delta_x = \sup_{x\in K} \big\| \big[\D,\pS(\cdot)\big](x) \big(\pS(x)\pm i\big)^{-1} \big\|$, 
fix $\lambda \geq \lambda_0$, 
and pick a compactly supported smooth function $f \in C_c^\infty(M)$ such that $f(x)^2 \geq \epsilon + \frac12 (\lambda^2 + \delta_K^2)$ for all $x\in K$. 
Inserting $\alpha = \lambda^{1/2} \delta_K^{-1}$ into \eqref{eq:comm_est}, we see that for any $x\in K$ we have 
\begin{align*}
\pm2 \big\la \{\til\D,\til\pS(\cdot)\}(x) \psi(x) \bigmvert \psi(x) \big\ra 
&\leq \lambda \big\lla \til\pS(x) \psi(x) \big\rra + (\lambda + \lambda^{-1} \delta_K^2) \big\lla \psi(x) \big\rra
\end{align*}
and in particular 
\begin{align*}
\big\la \{\til\D,\til\pS(\cdot)\}(x) \psi(x) \bigmvert \psi(x) \big\ra 
&\geq -\frac12 \lambda \big\lla \til\pS(x) \psi(x) \big\rra - \frac12 (\lambda + \lambda^{-1} \delta_K^2) \big\lla \psi(x) \big\rra .
\end{align*}
Thus for any $x\in K$ we have 
\begin{align}
\label{eq:comm_est_K}
&\big\lla \lambda\til\pS(x) \psi(x) \big\rra + \big\la \{\til\D,\lambda\til\pS(\cdot)\}(x) \psi(x) \bigmvert \psi(x) \big\ra + \big\lla f(x) \psi(x) \big\rra \nonumber\\
&\qquad\geq \frac12 \big\lla \lambda \til\pS(x) \psi(x) \big\rra - \frac12 (\lambda^2 + \delta_K^2) \big\lla \psi(x) \big\rra + f(x)^2 \big\lla \psi(x) \big\rra 
\geq \epsilon \big\lla \psi(x) \big\rra .
\end{align}
\end{description}
Combining \cref{eq:comm_est_M-K,eq:comm_est_K}, we have thus shown the desired inequality for \emph{any} $x\in M$. 
\end{proof}

\begin{thm}
\label{thm:Fredholm}
\begin{enumerate}
\item 
There exists $\lambda_0>0$ such that for any $\lambda\geq\lambda_0$ the operator $\til\D_{\lambda\pS}$ is Fredholm, and thus $\D_{\lambda\pS}$ is a generalised Dirac--Schrödinger operator. 
\item 
Suppose there exists a compact subset $\hat K \supset K$ such that $\hat\delta < \frac{\hat c^2}{\hat c+1}$, where 
\begin{align*}
\hat\delta &:= \sup_{x\in M\setminus\hat K} \big\| \big[\D,\pS(\cdot)\big](x) \big(\pS(x)\pm i\big)^{-1} \big\| , &
\hat c &:= \inf_{x\in M\setminus\hat K} \|\pS(x)^{-1}\|^{-1} .
\end{align*}
Then the first statement holds with $\lambda_0=1$. 
In particular, $\D_\pS$ is a generalised Dirac--Schrödinger operator. 
\end{enumerate}
\end{thm}
\begin{proof}
Let $\lambda \geq \lambda_0$, $\epsilon>0$, and $f\in C_c^\infty(M)$ be given by \cref{lem:anti-comm_estimate}. 
For any $\psi \in \Dom(\til\D_{\lambda\pS}^2)$ we then compute 
\begin{align}
\label{eq:square_est_M}
\big\la \psi \bigmvert &\big( \til\D_{\lambda\pS}^2 + f^2 \big) \psi \big\ra 
= \big\la \til\D_{\lambda\pS} \psi \bigmvert \til\D_{\lambda\pS} \psi \big\ra + \big\la \psi \bigmvert f^2\psi \big\ra \nonumber\\
&= \big\lla \til\D \psi \big\rra + \big\lla \lambda\til\pS(\cdot) \psi \big\rra + \big\la \til\D \psi \bigmvert \lambda\til\pS(\cdot) \psi \big\ra + \big\la \lambda\til\pS(\cdot) \psi \bigmvert \til\D \psi \big\ra + \big\lla f\psi\big\rra \nonumber\\
&\geq \big\lla \lambda\til\pS(\cdot) \psi \big\rra + \big\la \big\{ \til\D , \lambda\til\pS(\cdot) \big\} \psi \bigmvert \psi \big\ra + \big\lla f\psi\big\rra \nonumber\\
&= \int_M \Big( \big\lla \lambda\til\pS(x) \psi(x) \big\rra + \big\la \big\{ \til\D , \lambda\til\pS(\cdot) \big\}(x) \psi(x) \bigmvert \psi(x) \big\ra + \big\lla f(x)\psi(x)\big\rra \Big) \dvol(x) \nonumber\\
&\geq \epsilon \int_M \lla\psi(x)\rra \dvol(x) 
= \epsilon \lla\psi\rra ,
\end{align}
where the inequality on the last line is given by \cref{lem:anti-comm_estimate}. 
Hence we have shown that the spectrum of $\til\D_{\lambda\pS}^2 + f^2$ is contained in $[\epsilon,\infty)$, and therefore we have a well-defined inverse $\big(\til\D_{\lambda\pS}^2 + f^2\big)^{-1} \in \mL_A\big(L^2(M,E\otimes\bF)^{\oplus2}\big)$. 

We can then construct a parametrix for $\til\D_{\lambda\pS}$ as follows. 
Pick a smooth function $\chi\in C_c^\infty(M)$ such that $0\leq\chi\leq1$, and $\chi(x)=1$ for all $x\in\supp f$. Write $\chi' := \sqrt{1-\chi^2}$. 
Using that $f\chi'=0$, we calculate that 
\begin{align*}
\til\D_{\lambda\pS} \chi' \til\D_{\lambda\pS} \big(\til\D_{\lambda\pS}^2+f^2\big)^{-1} \chi' 
&= [\til\D,\chi'] \til\D_{\lambda\pS} \big(\til\D_{\lambda\pS}^2+f^2\big)^{-1} \chi' + (\chi')^2 .
\end{align*}
Define the operator
\begin{align*}
Q &:= \chi \big(\til\D_{\lambda\pS}-i\big)^{-1} \chi + \chi' \til\D_{\lambda\pS} \big(\til\D_{\lambda\pS}^2+f^2\big)^{-1} \chi' . 
\end{align*}
We then compute 
\begin{align*}
\til\D_{\lambda\pS} Q - 1 
&= \big[\til\D,\chi\big] (\til\D_{\lambda\pS}-i)^{-1} \chi + i \chi (\til\D_{\lambda\pS}-i)^{-1} \chi + \big[\til\D,\chi'\big] \til\D_{\lambda\pS} \big(\til\D_{\lambda\pS}^2+f^2\big)^{-1} \chi' . 
\end{align*}
The operators $[\til\D,\chi]$ and $[\til\D,\chi']$ are smooth and compactly supported, 
and therefore bounded. Since $(\til\D_{\lambda\pS}-i) \big(\til\D_{\lambda\pS}^2+f^2\big)^{-\frac12}$ is also bounded, it follows from \cref{prop:cpt_res} that $\til\D_{\lambda\pS} Q - 1$ is compact. 
Hence $Q$ is a right parametrix for $\til\D_{\lambda\pS}$. A similar calculation shows that $Q$ is also a left parametrix, and therefore $\til\D_{\lambda\pS}$ is Fredholm. 
We have thus proven the first statement. 

For the second statement, we note that we may replace $K$ by the larger compact set $\hat K$. 
Using the inequality $\hat\delta < \frac{\hat c^2}{\hat c+1}$,
the proof of \cref{lem:anti-comm_estimate} (picking $\lambda_0=1$) shows that for all $x\in M\setminus\hat K$ we have 
\begin{align*}
\big\lla \til\pS(x) \psi(x) \big\rra + \big\la \{\til\D,\til\pS(\cdot)\}(x) \psi(x) \bigmvert \psi(x) \big\ra 
&\geq \epsilon \big\lla \psi(x) \big\rra ,
\end{align*}
for $\epsilon := \frac12 \big( \hat c^2 - \hat\delta^2 (1+\hat c^{-1})^2 \big) > 0$. 
Thus in this case, the first statement holds with $\lambda_0=1$. 
\end{proof}

\begin{prop}
\label{prop:D_S_invertible}
Suppose that $\{\pS(x)\}_{x\in M}$ is uniformly invertible on \emph{all} of $M$. 
Then there exists $\lambda_0>0$ such that for any $\lambda\geq\lambda_0$ the generalised Dirac--Schrödinger operator $\til\D_{\lambda\pS}$ is also invertible. 
\end{prop}
\begin{proof}
Since $\pS(\cdot)$ is uniformly invertible, \cref{eq:comm_est_M-K} now holds for all $x\in M$ (for $\lambda\geq\lambda_0>0$) and therefore \cref{eq:square_est_M} holds with $f\equiv0$, which shows that $\til\D_{\lambda\pS}^2$ (and hence $\til\D_{\lambda\pS}$) is invertible. 
\end{proof}

\subsection{The index pairing}
\label{sec:Kasp_prod}

In this subsection, we will prove \cref{thm:Kasp_prod_index}. 
Similarly to \cite[Proposition 5.14]{vdD19_Index_DS}, we show first that we can replace $M$ by a manifold with cylindrical ends, without affecting the index of the generalised Dirac--Schrödinger operator. 
\begin{prop}
\label{prop:cyl_ends}
There exist a precompact open subset $U$ of $M$ and a generalised Dirac--Schrödinger operator $\D_{\lambda\pS}'$ on $M' := \bar U \cup_{\partial U} (\partial U \times [0,\infty))$ satisfying assumptions \ref{ass:A} and \ref{ass:B}, such that 
\begin{enumerate}
\item \label{item1:cyl_ends}
the operators $\D'$ and $\pS'(\cdot)$ on $M'$ agree with $\D$ and $\pS(\cdot)$ on $M$ when restricted to $U$; 
\item \label{item2:cyl_ends}
the metric and the operators $\D'$ and $\pS'(\cdot)$ on $M'$ are of product form on $\partial U \times [1,\infty)$; 
\item 
we have the equality $\Index\big( \D' - i \lambda\pS'(\cdot) \big) = \Index\big( \D - i \lambda\pS(\cdot) \big) \in \K_0(A)$ for $\lambda$ sufficiently large. 
\end{enumerate}
In particular, $M'$ is complete and $\D'$ has bounded propagation speed. 
\end{prop}
\begin{proof}
The proof is similar to the proof of \cite[Proposition 5.14]{vdD19_Index_DS}, but requires some minor adaptations. For completeness, we include the details here. 

Let $U$ be a precompact open neighbourhood of $K$, with smooth compact boundary $\partial U$. 
Consider the manifold $M' := \bar U \cup_{\partial U} \big(\partial U\times[0,\infty)\big)$ with cylindrical ends. 
For some $0<\epsilon<1$, let $C\simeq\partial U\times(-\epsilon,\epsilon)$ be a tubular neighbourhood of $\partial U$, such that there exists a diffeomorphism $\phi\colon U\cup C \to \bar U\cup_{\partial U}\big(\partial U\times[0,\epsilon)\big) \subset M'$ (which preserves the subset $U$). 
Equip $M'$ with a Riemannian metric which is of product form on $\partial U\times[1,\infty)$ (ensuring that $M'$ is complete), and which agrees with $g|_U$ on $U$. 
Let $\bF'\to M'$ be a hermitian vector bundle which agrees with $\bF|_U$ on $U$. 
Let $\D'$ be a symmetric elliptic first-order differential operator on $\bF'\to M'$, which is of product form on $\partial U\times[1,\infty)$, and which agrees with $\D|_{U\cup C}$ on $U\cup C$. 
Then $\D'$ has bounded propagation speed and is essentially self-adjoint by \cite[Proposition 10.2.11]{Higson-Roe00}. 

Let $0<\delta<\epsilon$ and let $\chi \in C^\infty(\R)$ be such that $0 \leq \chi(r) \leq 1$ for all $r\in\R$, $\chi(r) = 1$ for all $r$ in a neighbourhood of $0$, and $\chi(r) = 0$ for all $|r| > \delta$. 
Consider the family $\{\pS'(x)\}_{x\in M'}$ given by
\[
\pS'(x) := \begin{cases}
            \pS(x) , & x\in U , \\
            \chi(r) \pS(x) + (1-\chi(r)) \pS(y) , & x=(r,y)\in [0,\infty)\times\partial U . 
            \end{cases}
\]
We choose $\delta$ small enough such that $\chi(r) \pS(y) + (1-\chi(r)) \pS(x)$ is invertible for all $x\in [0,\delta]\times\partial U$. 
Then the family $\{\pS'(x)\}_{x\in M'}$ also satisfies assumptions \ref{ass:A} and \ref{ass:B}. 
Thus we have constructed a Dirac--Schrödinger operator $\D' - i \lambda\pS'(\cdot)$ on $M'$, satisfying the desired properties \ref{item1:cyl_ends} and \ref{item2:cyl_ends}. 
It remains to prove that $\Index\big( \D' - i \lambda\pS'(\cdot) \big) = \Index\big( \D - i \lambda\pS(\cdot) \big)$, for which we invoke the relative index theorem. 

Let $M^1 := M$ and $M^2 := \partial U \times \R$. 
Let $C' = \phi(C)$ be the collar neighbourhood of $\partial U$ in $M'$. 
We equip $M^2$ with a complete Riemannian metric which agrees with the metric of $M'$ on $C' \cup \big(\partial U\times(0,\infty)\big)$, and which is of product form on $(-\infty,-1]\times\partial U$. We extend the vector bundle $\bF'|_{C'\cup (\partial U\times(0,\infty))}$ to a bundle $\bF^2\to M^2$, and we pick an essentially self-adjoint elliptic first-order differential operator $\D^2$ on $\bF^2$ such that $\D^2|_{C'\cup(\partial U\times(0,\infty))} = \D'|_{C'\cup(\partial U\times(0,\infty))}$ (for instance, we can take $\D^2$ to be of product form on $(-\infty,-1)\times\partial U$). 
We define a family $\{\pS^2(x)\}_{x\in M^2}$ by $\pS^2(y,r) := \chi(r) \pS(y,r) + (1-\chi(r)) \pS(y)$ for all $y\in\partial U$ and $r\in\R$. 
Then $\bF^2\to M^2$, $\D^2$, and $\pS^2(\cdot)$ satisfy assumptions \ref{ass:A} and \ref{ass:B}. By cutting and pasting along $\partial U$, we obtain manifolds $M^3 = M'$ and $M^4 = \big(\partial U\times(-\infty,0]\big) \cup_{\partial U} (M\setminus U)$, with corresponding operators $\D^3$, $\pS^3(\cdot)$, $\D^4$, and $\pS^4(\cdot)$. 
By \cref{thm:rel_index}, we have the equality $\Index \big( \D^1 - i \lambda\pS^1(\cdot) \big) + \Index \big( \D^2 - i \lambda\pS^2(\cdot) \big) = \Index \big( \D^3 - i \lambda\pS^3(\cdot) \big) + \Index \big( \D^4 - i \lambda\pS^4(\cdot) \big)$. 
The potentials $\pS^2(\cdot)$ and $\pS^4(\cdot)$ are both uniformly invertible, so by \cref{prop:D_S_invertible} we have $\Index \big( \D^2 - i \lambda\pS^2(\cdot) \big) = \Index \big( \D^4 - i \lambda\pS^4(\cdot) \big) = 0$ (for $\lambda$ sufficiently large). 
Since $M^1 = M$ and $M^3 = M'$, we conclude that $\Index \big( \D - i \lambda\pS(\cdot) \big) = \Index \big( \D' - i \lambda\pS'(\cdot) \big)$. 
\end{proof}

\begin{prop}[{cf.\ \cite[Proposition 5.10]{vdD19_Index_DS}}]
\label{prop:cyl_ends_Kasp_prod}
Let $M'$, $\D'$, and $\pS'(\cdot)$ be as in \cref{prop:cyl_ends}. 
Then we have the equality $[\pS'(\cdot)] \otimes_{C_0(M')} [\D'] = \Index \big( \D' - i \lambda\pS'(\cdot) \big)$. 
\end{prop}
\begin{proof}
This follows basically from \cite[Proposition 5.10]{vdD19_Index_DS}; however, there it was assumed that the graph norms of $\pS'(x)$ are \emph{uniformly} equivalent and that the weak derivative $d\pS'(\cdot)$ is \emph{uniformly} bounded. 
Together with the bounded propagation speed of $\D'$, this implies the boundedness of $\big[\D',\pS'(\cdot)\big] \big(\pS'(\cdot)\pm i\big)^{-1}$, as explained in \cref{remark:B}.\ref{item:bdd_comm}. 
The proofs given in \cite[\S5.1]{vdD19_Index_DS} actually only rely on this boundedness of $\big[\D',\pS'(\cdot)\big] \big(\pS'(\cdot)\pm i\big)^{-1}$. 
As the latter is required by our assumption \ref{ass:B2}, the proof of \cite[Proposition 5.10]{vdD19_Index_DS} follows through and the statement follows. 
\end{proof}

We are now ready to prove: 
\thmKaspprodindex*
\begin{proof}
From \cref{prop:cyl_ends}, we obtain a complete manifold $M'$ and a generalised Dirac--Schrödinger operator satisfying assumptions \ref{ass:A} and \ref{ass:B}, such that $\D'$ has bounded propagation speed, and such that $\Index\big( \D' - i \lambda\pS'(\cdot) \big) = \Index\big( \D - i \lambda\pS(\cdot) \big) \in \K_0(A)$ (for $\lambda$ sufficiently large).  
As in the proof of \cite[Theorem 5.15]{vdD19_Index_DS}, we have 
\[
[\pS(\cdot)] \otimes_{C_0(M)} [\D] = [\pS(\cdot)|_U] \otimes_{C_0(U)} [\D|_U] = [\pS'(\cdot)] \otimes_{C_0(M')} [\D'] . 
\]
Moreover, we know from \cref{prop:cyl_ends_Kasp_prod} that $[\pS'(\cdot)] \otimes_{C_0(M')} [\D'] = \Index\big( \D' - i \lambda\pS'(\cdot) \big)$. 
Altogether, we conclude that 
\[
[\pS(\cdot)] \otimes_{C_0(M)} [\D] 
= [\pS'(\cdot)] \otimes_{C_0(M')} [\D'] 
= \Index\big( \D' - i \lambda\pS'(\cdot) \big) 
= \Index\big( \D - i \lambda\pS(\cdot) \big) .
\qedhere
\]
\end{proof}

\section{Proof of the main theorem}
\label{sec:proof}

Let  $M$, $\D$, and $\pS(\cdot)$ satisfy assumptions \ref{ass:A}, \ref{ass:B}, and \ref{ass:C}, and consider the generalised Callias-type operator $\D_{\lambda\pS}$. 
We will show that we can replace the manifold $M$ by a cylindrical manifold $\R\times N$, without changing the index of $\D_{\lambda\pS}$. 
Thus we can reduce the proof of our generalised Callias Theorem (\cref{thm:general_Callias}) from the general statement to the case of a cylindrical manifold. 
This reduction is made possible by the relative index theorem (\cref{thm:rel_index}). 

\begin{lem}
\label{lem:collar_potential}
We may replace the collar neighbourhood $C$ by a smaller collar neighbourhood $C' \simeq (-2\varepsilon',2\varepsilon')\times N$ (with $0 < \varepsilon' < \varepsilon$) and the potential $\pS(\cdot)$ by a potential $\pS'(\cdot)$ satisfying:
\begin{itemize}
\item 
for all $x\in K\setminus C'$: $\pS'(x) = \pT$;
\item 
for all $x=(r,y) \in C'$: $\pS'(x) = \varrho(r) \pT + \big(1-\varrho(r)\big) \pS(y)$, 
for some function $\varrho \in C^\infty(\R)$ such that $0 \leq \varrho(r) \leq 1$ for all $r\in\R$, $\varrho(r)=1$ for all $r$ in a neighbourhood of $(-\infty,-\varepsilon']$, and $\varrho(r)=0$ for all $r$ in a neighbourhood of $[0,\infty)$, 
\end{itemize}
such that $[\pS(\cdot)] = [\pS'(\cdot)] \in \K_1(C_0(M,A))$ and (for $\lambda$ sufficiently large) $\Index\big( \D - i \lambda\pS(\cdot) \big) = \Index\big( \D - i \lambda\pS'(\cdot) \big) \in \K_0(A)$. 
\end{lem}
\begin{proof}
\begin{enumerate}
\item \label{item:step1}
In a first step, we replace $\pS(\cdot)$ by a potential $\pS''(\cdot)$ which is of `product form' near $N$. 
Let $0<\varepsilon'<\varepsilon/2$ and let $\chi \in C^\infty(\R)$ be such that $0 \leq \chi(r) \leq 1$ for all $r\in\R$, $\chi(r) = 1$ for all $|r|\leq2\varepsilon'$, and $\chi(r) = 0$ for all $|r| > 3\varepsilon'$. 
Consider the family $\{\pS''(x)\}_{x\in M}$ given by
\[
\pS''(x) := \begin{cases}
            \pS(x) , & x\in M\setminus C , \\
            \chi(r) \pS(y) + (1-\chi(r)) \pS(x) , & x=(r,y)\in C \simeq (-2\varepsilon,2\varepsilon)\times N . 
            \end{cases}
\]
We choose $\varepsilon'$ small enough such that $\chi(r) \pS(y) + (1-\chi(r)) \pS(x)$ is invertible for all $x\in [-3\varepsilon',3\varepsilon']\times N$. 
Then $\pS''(\cdot)$ satisfies $\pS''(r,y) = \pS(y) = \pS''(0,y)$ for all $x=(r,y)$ in the collar neighbourhood $C'\simeq (-2\varepsilon',2\varepsilon')\times N$ of $N$. 
Connecting $\pS(\cdot)$ and $\pS''(\cdot)$ via a straight-line homotopy, we obtain $[\pS''(\cdot)] = [\pS(\cdot)]$. 
Moreover, since $\pS''(\cdot)$ again satisfies assumptions \ref{ass:A} and \ref{ass:B}, 
we can apply \cref{thm:Kasp_prod_index}: there exist $\lambda_0,\lambda_0''$ such that for all $\lambda \geq \max\{\lambda_0,\lambda_0''\}$ we have 
\[
\Index\big( \D - i \lambda\pS(\cdot) \big) 
\stackrel{\lambda\geq\lambda_0}{=} [\pS(\cdot)] \otimes_{C_0(M)} [\D] 
= [\pS''(\cdot)] \otimes_{C_0(M)} [\D] 
\stackrel{\lambda\geq\lambda_0''}{=} \Index\big( \D - i \lambda\pS''(\cdot) \big) .
\]
\item 
Picking $\varrho$ as in the statement with $\varepsilon'$ from step \ref{item:step1}, we consider the potential $\pS'(\cdot)$ given by 
\begin{align}
\label{eq:collar_potential}
\pS'(x) := 
           \begin{cases}
           \pT , & x\in K\setminus C' , \\
           \varrho(r) \pT + \big(1-\varrho(r)\big) \pS(y) , & x=(r,y)\in C', \\
           \pS''(x) , & x\in M\setminus (K\cup C') . 
           \end{cases}
\end{align}
We note that $\pS'(\cdot)$ again satisfies assumption \ref{ass:A}, and therefore also defines a class $[\pS'(\cdot)] \in \K_1(C_0(M,A))$. 
The difference 
$\pS'(x) - \pS''(x)$ 
is relatively $\pS''(x)$-compact (by choice of $\pT$) and vanishes outside of $K$. 
Moreover, by assumption \ref{ass:A2} we know that $\big( \pS''(x) \pm i \big) \big( \pT \pm i \big)^{-1}$ is norm-continuous in $x$, and consequently the family of inverses $\big( \pT \pm i \big) \big( \pS''(x) \pm i \big)^{-1}$ is also norm-continuous in $x$. 
Therefore $\big( \pS'(x) - \pS''(x) \big) \big(\pS''(x)\pm i\big)^{-1}$ is norm-continuous in $x$. 
Hence the family $\pS'(\cdot) - \pS''(\cdot)$ is relatively $\pS''(\cdot)$-compact, and it follows from \cref{prop:KK-class_rel_cpt_pert} that $[\pS'(\cdot)] = [\pS''(\cdot)] \in \KK^1(\C,C_0(M,A)) \simeq \K_1(C_0(M,A))$. 

Next, since $\pS'(\cdot)$ again satisfies assumption \ref{ass:B}, 
we can again apply \cref{thm:Kasp_prod_index}, and as in step \ref{item:step1} we obtain (for $\lambda$ sufficiently large) that $\Index\big( \D - i \lambda\pS'(\cdot) \big) = \Index\big( \D - i \lambda\pS''(\cdot) \big)$. 
\qedhere
\end{enumerate}
\end{proof}

\begin{defn}
\label{defn:cylinder}
Consider the cylindrical manifold $\R\times N$, along with the pullback vector bundle $\bF_{\R\times N}$ obtained from $\bF_N\to N$. 
We identify $\Gamma_c^\infty(\bF_{\R\times N}) \simeq C_c^\infty(\R) \otimes \Gamma^\infty(\bF_N)$, 
and consider the essentially self-adjoint elliptic first-order differential operator $\D_{\R\times N}$ on $\bF_{\R\times N}$ given by 
\begin{equation*}
\D_{\R\times N} := -i\partial_r \otimes \Gamma_N + 1 \otimes \D_N .
\end{equation*}
Let $\varrho \in C^\infty(\R)$ be as in \cref{lem:collar_potential}. 
We define the family $\{\pS_{\R\times N}(r,y)\}_{(r,y)\in\R\times N}$ on $E$ given by
\begin{align*}
\pS_{\R\times N}(r,y) := \varrho(r) \pT + \big(1-\varrho(r)\big) \pS(y) .
\end{align*}
\end{defn}

The operator $\Gamma_N$ from assumption \ref{ass:C1} provides a $\Z_2$-grading on $\bF_N$, yielding the decomposition $\bF_N = \bF_N^+ \oplus \bF_N^-$. By assumption, the essentially self-adjoint elliptic first-order differential operator $\D_N$ is odd with respect to this $\Z_2$-grading, and thus $\D_N$ defines an even $\K$-homology class $[\D_N] \in \K^{0}(C(N)) \equiv \K_0(N)$. 
Similarly, the ungraded operator $\D_{\R\times N}$ yields an odd $\K$-homology class $[\D_{\R\times N}] \in \K^1(C_0(\R\times N)) \equiv \K_1(\R\times N)$. 
Furthermore, the operator $-i\partial_r$ on $L^2(\R)$ yields an odd $\K$-homology class $[-i\partial_r] \in \K^1(C_0(\R)) \equiv \K_1(\R)$. 

\begin{lem}
\label{lem:D_product}
The external product of $[-i\partial_r] \in \K^1(C_0(\R))$ with $[\D_N] \in \K^{0}(C(N))$ equals $[\D_{\R\times N}] \in \K^1(C_0(\R\times N))$. 
\end{lem}
\begin{proof}
The statement follows from the description of the odd-even (internal) Kasparov product given in \cite[Example 2.38]{BMS16} (noting that the argument remains valid in the simpler case of an external Kasparov product). 
\end{proof}

\begin{thm}
\label{thm:red_cyl}
Consider the cylindrical manifold $\R\times N$ with the operators $\D_{\R\times N}$ and $\pS_{\R\times N}(\cdot)$ from \cref{defn:cylinder}. 
Then, for $\lambda$ sufficiently large, 
\[
\Index\big( \D - i \lambda \pS(\cdot) \big) = \Index\big( \D_{\R\times N} - i \lambda \pS_{\R\times N}(\cdot) \big) .
\]
\end{thm}
\begin{proof}
Using \cref{lem:collar_potential}, we may assume that the potential $\pS(\cdot)$ agrees with the potential $\pS_{\R\times N}(\cdot)$ on the collar neighbourhood $C$ of $N$, and that $\pS(x)=\pT$ for all $x\in K\setminus C$. 
We will then apply the relative index theorem (\cref{thm:rel_index}) twice. 

First, define $V := M\setminus K$, and consider the manifolds 
\begin{align*}
M^1 &\equiv M = K \cup_{N} \bar V, & 
M^2 &:= \R\times N = \big((-\infty,0]\times N\big) \cup_{\{0\}\times N} \big([0,\infty)\times N\big). 
\end{align*}
On $M^2$, we consider the operator $\D^2 := \D_{\R\times N}$ and the potential $\pS^2(r,y) := \pS(y)$, satisfying assumptions \ref{ass:A} and \ref{ass:B}. 
We identify $N$ in $M$ with $\{0\}\times N$ in $M^2$. 
Cutting and pasting then gives us two new manifolds $M^3 = K \cup_{N} \big([0,\infty)\times N\big)$ and $M^4 = \big((-\infty,0]\times N\big) \cup_{N} \bar V$. 
By the relative index theorem, we have $\Index\big( \D^1 - i \lambda\pS^1(\cdot) \big) + \Index\big( \D^2 - i \lambda\pS^2(\cdot) \big) = \Index\big( \D^3 - i \lambda\pS^3(\cdot) \big) + \Index\big( \D^4 - i \lambda\pS^4(\cdot) \big) \in \K_0(A)$ (for $\lambda$ sufficiently large). 
Since $\pS^2(\cdot)$ and $\pS^4(\cdot)$ are invertible, we know from \cref{prop:D_S_invertible} that also $\D^2 - i \lambda\pS^2(\cdot)$ and $\D^4 - i \lambda\pS^4(\cdot)$ are invertible (for $\lambda$ sufficiently large), and it follows that $\Index\big( \D - i \lambda\pS(\cdot) \big) \equiv \Index\big( \D^1 - i \lambda\pS^1(\cdot) \big) = \Index\big( \D^3 - i \lambda\pS^3(\cdot) \big) \in \K_0(A)$. 
Thus we have replaced the subset $V$ by the half cylinder $(0,\infty)\times N$, with the potential $\pS_{\R\times N}(r,y) = \pS(y)$ for all $r\in(0,\infty)$. 

Second, we can similarly apply the relative index theorem again to replace the subset $K \backslash \big([-\varepsilon,0]\times N\big)$ by the half cylinder $(-\infty,-\varepsilon)\times N$, equipped with the constant potential $\pS_{\R\times N}(r,y) = \pT$ for all $r\in(-\infty,-\varepsilon)$, $y\in N$. 
This completes the proof. 
\end{proof}

We are now ready to prove our main theorem. 

\thmGeneralCallias*
\begin{proof}
Consider the cylindrical manifold $\R\times N$ with the operators $\D_{\R\times N}$ and $\pS_{\R\times N}(\cdot)$ from \cref{defn:cylinder}. 
From \cref{prop:sf_rel_cpt_family} we have the equality 
\[
\relind\big(P_+(\pS_N(\cdot)),P_+(\pT(\cdot))\big) = \SF \big( \{\pS_{\R\times N}(r)\}_{r\in[-\epsilon,0]} \big) .
\]
Moreover, by \cite[Proposition 2.21]{vdD19_Index_DS}, the spectral flow of the family $\{\pS_{\R\times N}(r)\}_{r\in[-\epsilon,0]}$ equals $[\pS_{\R\times N}(\cdot)] \otimes_{C_0(\R)} [-i\partial_r]$, so we obtain 
\[
\relind\big(P_+(\pS_N(\cdot)),P_+(\pT(\cdot))\big) = [\pS_{\R\times N}(\cdot)] \otimes_{C_0(\R)} [-i\partial_r] .
\]

For $C^*$-algebras $A$, $B$, and $C$, recall the map $\tau_C\colon \KK(A,B)\to\KK(A\otimes C,B\otimes C)$ given by the external Kasparov product with the identity element $1_C\in\KK(C,C)$. 
Applying this with $C=C(N)$, we then have the equalities 
\begin{align*}
\relind\big(P_+(\pS_N(\cdot)),P_+(\pT(\cdot))\big) \otimes_{C(N)} [\D_N] 
&= \Big( [\pS_{\R\times N}(\cdot)] \otimes_{C_0(\R)} [-i\partial_r] \Big) \otimes_{C(N)} [\D_N] \\
&= [\pS_{\R\times N}(\cdot)] \otimes_{C_0(\R\times N)} \tau_{C(N)}([-i\partial_r]) \otimes_{C(N)} [\D_N] \\
&= [\pS_{\R\times N}(\cdot)] \otimes_{C_0(\R\times N)} \Big( [-i\partial_r] \otimes [\D_N] \Big) \\
&= [\pS_{\R\times N}(\cdot)] \otimes_{C_0(\R\times N)} [\D_{\R\times N}] ,
\end{align*}
where 
the second and third equalities follow from the properties of the Kasparov product, 
and the fourth equality is given by \cref{lem:D_product}. 

Since the operators $\pS_{\R\times N}(\cdot)$ and $\D_{\R\times N}$ on the manifold $\R\times N$ satisfy the assumptions \ref{ass:A} and \ref{ass:B}, 
we may apply \cref{thm:Kasp_prod_index} to compute the Kasparov product on the manifold $\R\times N$ and obtain 
\[
[\pS_{\R\times N}(\cdot)] \otimes_{C_0(\R\times N)} [\D_{\R\times N}]
= \Index\big( \D_{\R\times N} - i \lambda \pS_{\R\times N}(\cdot) \big) . 
\]
Finally, from \cref{thm:red_cyl} we know that 
\[
\Index\big( \D_{\R\times N} - i \lambda \pS_{\R\times N}(\cdot) \big) = \Index\big( \D - i \lambda \pS(\cdot) \big) .
\qedhere
\]
\end{proof}

\appendix 

\section{Appendix}

In this Appendix we collect several statements regarding (mostly unbounded) operators on Hilbert $C^*$-modules. 
Many of these statements are well-known for operators on Hilbert spaces, but they have not yet appeared (to the author's best knowledge) in the literature for operators on Hilbert $C^*$-modules. 
While some of the proofs of these statements are similar to proofs in the Hilbert space context, we have for completeness included detailed proofs in our Hilbert $C^*$-module context. 

Throughout this Appendix, we consider a $C^*$-algebra $A$ and a Hilbert $A$-module $E$. 
We start with a basic lemma (well-known in the Hilbert space setting), whose proof given in e.g. \cite[Theorem 9.19]{Hislop-Sigal12} remains valid for adjointable operators on Hilbert $C^*$-modules.
\begin{lem}
\label{lem:s-limit_times_cpt}
Let $T_n \xrightarrow{*s} T \in \mL_A(E)$ be a $*$-strongly convergent sequence of adjointable operators on $E$. 
Then for any compact operator $K \in \mK_A(E)$, we have norm-convergence $KT_n \to KT$ and $T_nK \to TK$. 
\end{lem}
\begin{proof}
For any $\epsilon>0$, there exists a finite-rank operator $F_\epsilon = \sum_{j=1}^N \theta_{\psi_j,\varphi_j}$ such that $\|K-F_\epsilon\| < \epsilon$. 
For each $\xi\in E$ we can estimate 
\[
\| (T_n-T) F_\epsilon \xi \| 
= \| (T_n-T) \sum_{j=1}^N \psi_j \la \varphi_j | \xi \ra \| 
\leq \sum_{j=1}^N \| (T_n-T) \psi_j \| \, \|\varphi_j\| \, \|\xi\| .
\]
Since $T_n\psi_j$ converges to $T\psi_j$ for each $j$, we obtain for $n$ large enough that $\| (T_n-T) F_\epsilon \| < \epsilon$. 
Furthermore, since $T_n$ converges strongly to $T$, the uniform boundedness principle implies that there exists $M\in(0,\infty)$ such that $\|T_n\| \leq M$ for all $n$. 
Thus for $n$ large enough, we obtain 
\[
\|T_nK-TK\| 
\leq \|(T_n-T)(K-F_\epsilon)\| + \|(T_n-T)F_\epsilon\| 
\leq \epsilon \|T_n-T\| + \epsilon 
\leq \epsilon (M+\|T\|+1) .
\]
As $\epsilon>0$ was arbitrary, this proves $T_nK$ converges to $TK$ in norm. 

Next, using that $T_n,T$ are adjointable, we can also estimate, for any $\xi\in E$, 
\[
\| F_\epsilon (T_n-T) \xi \| 
= \| \sum_{j=1}^N \psi_j \la (T_n^*-T^*) \varphi_j | \xi \ra \| 
\leq \sum_{j=1}^N \|\psi_j\| \, \| (T_n^*-T^*) \varphi_j \| \, \|\xi\| .
\]
By assumption, $T_n^*$ also converges strongly to $T^*$, so for $n$ large enough we obtain that $\| F_\epsilon (T_n-T) \| < \epsilon$.
Then, proceeding as above, also $KT_n$ converges to $KT$ in norm. 
\end{proof}

\subsection{Interpolation}

The following results are based on \cite[Proposition A.1]{Les05}. 
We follow the adaptation to the case of operators on Hilbert $C^*$-modules as given in the proof of \cite[Lemma 7.7]{LM19}. 
\begin{prop}
\label{prop:interpolation}
Let $T$ be an invertible positive regular self-adjoint operator on $E$. 
Let $S$ be a densely defined symmetric operator on $E$ with $\Dom S \supset \Dom T$. 
Then the following statements hold:
\begin{enumerate}
\item 
$ST^{-1}$ is bounded and adjointable, and 
$T^{-1}S$ is densely defined and bounded and extends to an adjointable operator $\bar{T^{-1}S}$ with $(\bar{T^{-1}S})^* = ST^{-1}$. 
\item 
\label{item:adjoint}
The operator $T^{-1}ST$ is densely defined and its adjoint equals $(T^{-1}ST)^* = TST^{-1}$. 
\item 
\label{item:bdd_adj}
If $T^{-1}ST$ or $TST^{-1}$ is bounded and extends to an adjointable operator, then in fact both $T^{-1}ST$ and $TST^{-1}$ are bounded and extend to adjointable operators, and $\|\bar{TST^{-1}}\| = \|\bar{T^{-1}ST}\|$. 
\item 
\label{item:square_roots}
The operator $T^{-\frac12}ST^{-\frac12}$ is bounded and extends to an adjointable operator, and its norm satisfies the inequality $\| T^{-\frac12}ST^{-\frac12} \| \leq \|ST^{-1}\|$. 
\end{enumerate}
\end{prop}
\begin{proof}
\begin{enumerate}
\item 
Since $S$ is closable and $\Ran T^{-1} \subset \Dom S$, it is a consequence of the closed graph theorem that $ST^{-1}$ is bounded. 
For all $\psi\in\Dom S$ and $\xi\in E$ we have $\la T^{-1}S\psi|\xi\ra = \la\psi|ST^{-1}\xi\ra$, which shows that $T^{-1}S$ has a densely defined adjoint and is therefore closable. Moreover, on the dense subset $\Dom S$, $T^{-1}S$ agrees with the adjoint $(ST^{-1})^*$. Thus $ST^{-1}$ is bounded and has a densely defined adjoint, which implies that $ST^{-1}$ is in fact adjointable with $(ST^{-1})^* = \bar{T^{-1}S}$. 
\item 
Since $T$ is regular and self-adjoint, $\Dom T^2$ is dense in $E$, so $\Dom(T^{-1}ST) = \Dom(ST) \supset \Dom(T^2)$ is also dense. 
Let $\xi\in\Dom(TST^{-1})$ and $\eta\in\Dom(T^{-1}ST)$. Then $T^{-1}\xi\in\Dom S$ with $ST^{-1}\xi\in\Dom T$, and $\eta\in\Dom T$ with $T\eta\in\Dom S$. 
Consequently, 
\[
\la TST^{-1}\xi | \eta \ra 
= \la ST^{-1}\xi | T\eta \ra 
= \la T^{-1}\xi | ST\eta \ra 
= \la \xi | T^{-1}ST\eta \ra ,
\]
so $\xi \in \Dom(T^{-1}ST)^*$ and $TST^{-1} \subset (T^{-1}ST)^*$. 
For the converse, consider $\xi \in \Dom(T^{-1}ST)^*$ and $\eta \in \Dom(T^2) \subset \Dom(T^{-1}ST)$. 
Then 
\[
\la (T^{-1}ST)^*\xi | \eta \ra 
= \la \xi | T^{-1}ST\eta \ra 
= \la T^{-1}\xi | ST\eta \ra 
= \la ST^{-1}\xi | T\eta \ra .
\]
Since $\Dom T^2$ is a core for $T$, the above equality continues to hold for all $\eta\in\Dom T$. Thus $ST^{-1}\xi \in \Dom T^* = \Dom T$ and $TST^{-1}\xi = T^*ST^{-1}\xi = (T^{-1}ST)^*\xi$, which shows $(T^{-1}ST)^* \subset TST^{-1}$. 
\item 
Assuming $\bar{T^{-1}ST}$ is adjointable, it follows from \ref{item:adjoint} that $TST^{-1} = (\bar{T^{-1}ST})^*$ is also bounded and adjointable. 
Similarly, assuming $\bar{TST^{-1}}$ is adjointable, it follows from \ref{item:adjoint} that $\bar{T^{-1}ST} = (\bar{TST^{-1}})^*$ is also bounded and adjointable. 
\item 
For simplicity, we assume that $\|T^{-1}\| \leq 1$. 
For $\xi,\eta\in\Dom T$ and $0\leq\Re z\leq1$, consider the operator $P_z := T^{-z} S T^{-1+z}$ and the function 
\[
f(z) := \la P_z \xi | \eta \ra = \la T^{-z} S T^{-1+z} \xi | \eta \ra .
\]
$f$ is weakly holomorphic on the strip $0<\Re z<1$. Moreover, from the estimate 
\[
\lla P_z \xi \rra 
\leq \lla ST^{-1+z} \xi \rra 
\leq \|ST^{-1}\|^2 \lla T^z \xi \rra 
\leq \|ST^{-1}\|^2 \lla T \xi \rra 
\]
we obtain that $\|f(z)\| \leq \|ST^{-1}\| \, \|T\xi\| \, \|\eta\|$, so $f$ is a bounded function. 

Now consider a bounded linear functional $\varphi \colon A \to \C$ with $\|\varphi\|\leq1$. 
Since the function $\varphi \circ f$ is holomorphic and bounded on the strip $0\leq\Re z\leq1$, it follows from the Hadamard 3-line Theorem that $\varphi\circ f$ is bounded by its suprema on the boundary $\Re z \in \{0,1\}$. 
On this boundary $\Re z \in \{0,1\}$ we have $\|P_z\| = \|P_0\| = \|ST^{-1}\|$, so from the Hadamard 3-line Theorem we obtain for all $0\leq\Re z\leq1$ that 
\begin{align*}
\big| \varphi \big( f(z) \big) \big| 
&= \big| \varphi \big( \la P_z\xi | \eta \ra \big) \big| 
\leq \sup_{w\in\C : \Re w=0,1} \big| \varphi \big( \la P_w\xi | \eta \ra \big) \big| \\
&\leq \sup_{w\in\C : \Re w=0,1} \big\| \la P_w\xi | \eta \ra \big\| 
\leq \|ST^{-1}\| \, \|\xi\| \, \|\eta\| .
\end{align*}
Since there exists a bounded linear functional $\varphi$ with $\big| \varphi \big( f(z) \big) \big| = \| f(z) \|$, it follows that also $\big\| f(z) \big\| \leq \|ST^{-1}\| \, \|\xi\| \, \|\eta\|$. 
Taking the supremum over all $\xi$ and $\eta$ with $\|\xi\|=\|\eta\|=1$, we conclude that $P_z$ is bounded and extends to an adjointable operator $\bar{P_z}$ satisfying $\|\bar{P_z}\| \leq \|ST^{-1}\|$. 
\qedhere
\end{enumerate}
\end{proof}

\begin{coro}
\label{coro:interpolation}
Let $T$ be an invertible positive regular self-adjoint operator on $E$. 
Let $F=F^*\in\mL_A(E)$ and assume that the operator $T^{-\frac12} F T^{\frac12}$ is bounded and extends to an adjointable operator. 
Then 
\[
\|F\| \leq \| \bar{T^{-\frac12} F T^{\frac12}} \| = \| \bar{T^{\frac12} F T^{-\frac12}} \| . 
\]
\end{coro}
\begin{proof}
We note that $T^{\frac12}$ is also an invertible positive regular selfadjoint operator. 
Certainly $\Dom F = E \supset \Dom T^{\frac12}$, and since $T^{-\frac12} F T^{\frac12}$ is bounded (and extends to an adjointable operator), we know from \cref{prop:interpolation}.\ref{item:bdd_adj} that also $T^{\frac12} F T^{-\frac12}$ is bounded (and extends to an adjointable operator) and $\|T^{\frac12} F T^{-\frac12}\| = \|T^{-\frac12} F T^{\frac12}\|$. 

Now consider the symmetric operator $S := T^{\frac12} F T^{\frac12}$. 
We have just seen that $ST^{-1} = T^{\frac12} F T^{-\frac12}$ is bounded (i.e., $\Dom T \subset \Dom S$). 
Hence by \cref{prop:interpolation}.\ref{item:square_roots} we find that 
\[
\|F\| 
= \| T^{-\frac12} S T^{-\frac12} \| 
\stackrel{\ref{prop:interpolation}.\ref{item:square_roots}}{\leq} \|ST^{-1}\| 
= \| T^{\frac12} F T^{-\frac12} \| 
\stackrel{\ref{prop:interpolation}.\ref{item:bdd_adj}}{=} \| T^{-\frac12} F T^{\frac12} \| .
\qedhere 
\]
\end{proof}

\subsection{Convergence of unbounded operators}

The following result generalises one of the statements in \cite[Proposition 2.2]{Les05}, regarding convergence of unbounded operators with respect to certain topologies, to the context of regular operators on Hilbert $C^*$-modules. 
\begin{prop}
\label{prop:Riesz_convergence}
Let $\D$ be a regular selfadjoint operator on $E$. 
We view the domain $W := \Dom\D \subset E$ as a Hilbert $A$-module with the graph norm. 
Let $T$ and $T_n$ (for all $n\in\N$) be regular selfadjoint operators on $E$ with $\Dom T = \Dom \D$ and $\Dom T_n = \Dom \D$ for all $n\in\N$. 

If $(T_n-T)(\D+i)^{-1}$ converges in norm to $0$ as $n\to\infty$, then also $F_{T_n}-F_T$ converges in norm to $0$ as $n\to\infty$. 
\end{prop}
\begin{proof}
The proof is similar to the Hilbert space proof given in \cite[Proposition 2.2]{Les05}. For completeness, we include the details here. 

We pick $0<\epsilon<\frac12$. 
Since $\Dom\D = \Dom T$, we know that $(\D+i)(T+i)^{-1}$ is bounded, so there exists an $N_0\in\N$ such that for all $n\geq N_0$ we have 
\[
\big\| (T-T_n) (T+i)^{-1} \big\| \leq \epsilon 
\quad\text{and}\quad 
\big\| (T+i)^{-1} (T-T_n) \big\| \leq \epsilon .
\]
Consequently, the operator $(T+i)^{-1}(T_n+i) = 1 - (T+i)^{-1}(T-T_n)$ is invertible and from the Neumann series we obtain the norm bound 
\[
\big\| (T_n+i)^{-1} (T+i) \big\| \leq \sum_{k=0}^\infty \epsilon^k = \frac{1}{1-\epsilon} < 2 .
\]
Thus for all $\psi\in E$ we have (in the $C^*$-algebra $A$) the inequality 
\[
\big\lla (T_n+i)^{-1} \psi \big\rra 
\leq \frac{1}{(1-\epsilon)^2} \big\lla (T+i)^{-1} \psi \big\rra ,
\]
or equivalently we have the operator inequality 
\[
(1+T_n^2)^{-1} \leq \frac{1}{(1-\epsilon)^2} (1+T^2)^{-1} .
\]
Furthermore, from the estimate $\big\| (T+i)^{-1} (T_n+i) \big\| \leq 1+\epsilon$ we similarly obtain the operator inequality 
\[
(1+T^2)^{-1} \leq (1+\epsilon)^2 (1+T_n^2)^{-1} .
\]
Thus, taking square roots, we have 
\[
\frac{1}{1+\epsilon} (1+T^2)^{-\frac12} \leq (1+T_n^2)^{-\frac12} \leq \frac{1}{1-\epsilon} (1+T^2)^{-\frac12} .
\]
Subtracting $(1+T^2)^{-\frac12}$ yields 
\[
-\frac{\epsilon}{1+\epsilon} (1+T^2)^{-\frac12} \leq (1+T_n^2)^{-\frac12} - (1+T^2)^{-\frac12} \leq \frac{\epsilon}{1-\epsilon} (1+T^2)^{-\frac12} ,
\]
from which we obtain the norm estimate 
\[
\big\| (1+T^2)^{\frac14} (1+T_n^2)^{-\frac12} (1+T^2)^{\frac14} - 1 \big\| \leq \frac{\epsilon}{1-\epsilon} .
\]
In particular, $(1+T^2)^{\frac14} (1+T_n^2)^{-\frac12} (1+T^2)^{\frac14}$ is bounded. Since also $(1+T^2)^{-\frac14} T_n (1+T^2)^{-\frac14}$ is bounded by \cref{prop:interpolation}.\ref{item:square_roots}, the estimate 
\[
\big\| (1+T^2)^{-\frac14} F_{T_n} (1+T^2)^{\frac14} \big\| 
\leq \big\| (1+T^2)^{-\frac14} T_n (1+T^2)^{-\frac14} \big\| \; \big\| (1+T^2)^{\frac14} (1+T_n^2)^{-\frac12} (1+T^2)^{\frac14} \big\| 
\]
shows that $(1+T^2)^{-\frac14} F_{T_n} (1+T^2)^{\frac14}$ is bounded. Thus we can use \cref{coro:interpolation} to estimate the difference of the bounded transforms of $T$ and $T_n$ (for all $n\geq N_0$):
\begin{align*}
\big\| F_T - F_{T_n} \big\| 
&\mathmakebox[2.5em][c]{\stackrel{\ref{coro:interpolation}}{\leq}} \big\| (1+T^2)^{-\frac14} \big( F_T - F_{T_n} \big) (1+T^2)^{\frac14} \big\| \\
&\mathmakebox[2.5em][c]{\leq} \big\| (1+T^2)^{-\frac14} \big( T - T_n \big) (1+T^2)^{-\frac14} \big\| \\
&\qquad\quad+ \big\| (1+T^2)^{-\frac14} T_n \big( (1+T^2)^{-\frac12} - (1+T_n^2)^{-\frac12} \big) (1+T^2)^{\frac14} \big\| \\
&\mathmakebox[2.5em][c]{\stackrel{\ref{prop:interpolation}.\ref{item:square_roots}}{\leq}} \big\| \big( T - T_n \big) (1+T^2)^{-\frac12} \big\| \\
&\qquad\quad+ \big\| (1+T^2)^{-\frac14} T_n (1+T^2)^{-\frac14} \big\| \; \big\| 1 - (1+T^2)^{\frac14} (1+T_n^2)^{-\frac12} (1+T^2)^{\frac14} \big\| \\
&\mathmakebox[2.5em][c]{\stackrel{\ref{prop:interpolation}.\ref{item:square_roots}}{\leq}} \epsilon + \big\| T_n (1+T^2)^{-\frac12} \big\| \frac{\epsilon}{1-\epsilon} \\
&\mathmakebox[2.5em][c]{\leq} \epsilon + (1+\epsilon) \frac{\epsilon}{1-\epsilon} 
= \epsilon \Big( 1 + \frac{1+\epsilon}{1-\epsilon} \Big) 
\leq 4\epsilon ,
\end{align*}
where we used that $\epsilon<\frac12$. 
We note that this inequality still holds for all $n\geq N_0$. Since $0<\epsilon<\frac12$ was arbitrary, this proves that $\big\| F_T - F_{T_n} \big\| \to 0$ as $n\to\infty$. 
\end{proof}

\subsection{Relatively compact perturbations}
\label{app:rel-cpt}

In this subsection we study relatively compact perturbations of regular self-adjoint operators. \cref{prop:rel_cpt_rel_bdd_0,prop:rel_cpt_pert_reg_sa} below are well-known facts for operators on Hilbert spaces, but appear not to be present in the literature on Hilbert $C^*$-modules. 

\begin{defn}
\label{defn:rel-cpt}
Let $T$ be a regular self-adjoint operator on $E$. 
A densely defined operator $R$ on $E$ is called \emph{relatively $T$-compact} if $\Dom(T) \subset \Dom(R)$ and $R(T\pm i)^{-1}$ is compact. 
\end{defn}
The assumption $\Dom(T) \subset \Dom(R)$ implies that $R$ is also relatively $T$-bounded. 
In fact, relative $T$-compactness implies that the relative $T$-bound can be chosen to be arbitrarily small. This is a well-known fact for operators on Hilbert spaces; we show next that this fact remains true on Hilbert $C^*$-modules, by adapting the proof of \cite[Theorem 14.2]{Hislop-Sigal12}. 

\begin{prop}
\label{prop:rel_cpt_rel_bdd_0}
Let $T$ be a regular self-adjoint operator on $E$, and let $R$ be relatively $T$-compact. 
Then for all $\epsilon>0$ there exists $C_\epsilon\geq0$ such that for all $\psi\in\Dom(T)$ we have 
\[
\| R\psi \| \leq \epsilon \|T\psi\| + C_\epsilon \|\psi\| .
\]
\end{prop}
\begin{proof}
We note that the operators $(T-i)(T-in)^{-1}$ converge $*$-strongly to $0$ as $n\to\infty$ (this follows e.g.\ from \cite[Lemma 7.2]{KL12}). 
We can write 
\[
R (T-in)^{-1} = R(T-i)^{-1} (T-i)(T-in)^{-1} .
\]
Since $R(T-i)^{-1}$ is compact and $(T-i)(T-in)^{-1} \xrightarrow{*s} 0$, the operator $R (T-in)^{-1}$ converges to $0$ in norm by \cref{lem:s-limit_times_cpt}. 
Thus, given any $\epsilon>0$, we can choose $n$ large enough such that $\|R (T-in)^{-1}\| < \epsilon$. 
Then for any $\psi\in\Dom(T)$ we have 
\[
\|R\psi\| 
\leq \|R (T-in)^{-1}\| \, \|(T-in)\psi\| 
\leq \epsilon \|(T-in)\psi\| 
\leq \epsilon \|T\psi\| + \epsilon n \|\psi\| ,
\]
where $C_\epsilon := \epsilon n$ is independent of $\psi$. 
\end{proof}

\begin{prop}
\label{prop:rel_cpt_pert_reg_sa}
Let $T$ be a regular self-adjoint operator on $E$, and let $R$ be a symmetric operator on $E$ which is relatively $T$-compact. 
Then $T+R$ is also regular and self-adjoint on $\Dom(T+R)=\Dom(T)$. 
\end{prop}
\begin{proof}
By \cref{prop:rel_cpt_rel_bdd_0}, we have for any $0<a<1$ that $\| R\psi \| \leq a \|T\psi\| + C_a \|\psi\|$ for all $\psi\in\Dom(T)$. It then follows from the Kato-Rellich Theorem on Hilbert $C^*$-modules (\cite[Theorem 4.5]{KL12}) that $T+R$ is also regular and self-adjoint with $\Dom(T+R)=\Dom(T)$. 
\end{proof}

The following result generalises \cite[Proposition 3.4]{Les05} to the context of regular operators on Hilbert $C^*$-modules. 
\begin{prop}
\label{prop:cpt_diff_bdd-transforms}
Let $T$ be a regular self-adjoint operator on $E$, and let $R$ be a symmetric operator on $E$ which is relatively $T$-compact. 
Then $F_{T+R} - F_T$ is compact. 
\end{prop}
\begin{proof}
The proof is similar to the Hilbert space proof given in \cite[Proposition 3.4]{Les05}, but requires minor adaptations. For completeness, we include the details here. 

\begin{enumerate}
\item 
We first prove a special case: assume that $R$ is compact and $\Ran R \subset \Dom T$. 
In this case it suffices to show the compactness of $F_{T+R} - F_T - R\big(1+(T+R)^2\big)^{-\frac12} = T\big(1+(T+R)^2\big)^{-\frac12} - T\big(1+T^2\big)^{-\frac12}$. 
We note that $(T+R)^2-T^2 = TR+R(T+R)$ is well-defined on $\Dom T = \Dom (T+R)$. 
We can then use the integral formula $(1+T^2)^{-\frac12} = \frac1\pi \int_0^\infty \lambda^{-\frac12} (1+\lambda+T^2)^{-1} \dx[\lambda]$ (and similarly for $T+R$) along with the resolvent identity to rewrite 
\begin{align*}
&T\big(1+(T+R)^2\big)^{-\frac12} - T\big(1+T^2\big)^{-\frac12} \\
&\quad= \frac1\pi \int_0^\infty \lambda^{-\frac12} T \big(1+\lambda+T^2\big)^{-1} \big( T^2 - (T+R)^2 \big) \big( 1+\lambda+(T+R)^2 \big)^{-1} \dx[\lambda] \\
&\quad= - \frac1\pi \int_0^\infty \lambda^{-\frac12} T \big(1+\lambda+T^2\big)^{-1} \big( TR + R(T+R) \big) \big( 1+\lambda+(T+R)^2 \big)^{-1} \dx[\lambda] .
\end{align*}
Observing that the integrand is compact and of order $\mO(\lambda^{-\frac32})$, 
we see that the integral converges in norm to a compact operator, and we conclude that $F_{T+R} - F_T$ is compact. 

\item 
We now prove the general case by reducing to the special case. 
For $n\in\N$, consider compactly supported continuous functions $\phi_n \in C_c(\R)$ satisfying $0 \leq \phi_n(x) \leq 1$ for all $x\in\R$, $\phi_n(x) = 1$ if $|x| \leq n$, and $\phi_n(x) = 0$ if $|x| \geq n+1$. 
Using continuous functional calculus, we construct the operators $\phi_n(T) \in \mL_A(E)$, and we note that $\Ran \phi_n(T) \subset \Dom T$ (since $\phi_n$ is compactly supported). 
We now consider the operators 
\[
R_n := \phi_n(T) R \phi_n(T) = \underbrace{\phi_n(T)}_{\text{bounded}} \, \underbrace{R(T-i)^{-1}}_{\text{compact}} \; \underbrace{(T-i)\phi_n(T)}_{\text{bounded}} ,
\]
and we note that $R_n$ is compact with $\Ran R_n \subset \Dom T$. Thus the special case applies to $R_n$. 

As $n\to\infty$, the operators $\phi_n(T)$ converge strongly (hence, by self-adjointness, $*$-strongly) to the identity (see e.g.\ \cite[Lemma 7.2]{KL12}). 
Since $R(T-i)^{-1}$ is compact, it follows from \cref{lem:s-limit_times_cpt} that $R_n(T-i)^{-1}$ converges in norm to $R(T-i)^{-1}$. 
From \cref{prop:Riesz_convergence} we therefore obtain that $F_{T+R_n}$ converges in norm to $F_{T+R}$. 
Since $F_{T+R_n}-F_T$ is compact by the special case, we conclude that also $F_{T+R}-F_T$ is compact. 
\qedhere 
\end{enumerate}
\end{proof}

\begin{prop}
\label{prop:cpt_diff_cts_funct}
Let $T$ be a regular self-adjoint operator on $E$, and let $R$ be a symmetric operator on $E$ which is relatively $T$-compact. 
Let $f\in C(\R)$ be a continuous function for which the limits $\lim_{x\to\pm\infty}f(x)$ exist. 
Then $f(T+R)-f(T)$ is compact. 
\end{prop}
\begin{proof}
The statement clearly holds for constant functions, and by \cref{prop:cpt_diff_bdd-transforms} also for the `bounded transform function' $b\in C(\R)$ given by $b(x) := x(1+x^2)^{-\frac12}$. It remains to prove the statement for functions $f\in C_0(\R)$ vanishing at infinity, and for this it suffices to consider $f(x) = (x\pm i)^{-1}$. But for the latter, the statement follows immediately from the resolvent identity and compactness of $R(T\pm i)^{-1}$:
\[
(T+R\pm i)^{-1} - (T\pm i)^{-1} = - (T+R\pm i)^{-1} R (T\pm i)^{-1} .
\qedhere
\]
\end{proof}

The following result partly generalises \cite[Corollary 3.5]{Les05} to the context of regular operators on Hilbert $C^*$-modules, under somewhat stronger assumptions. 
\begin{coro}
\label{coro:cpt_diff_spec-projs}
Let $T$ be a regular self-adjoint operator on $E$, and let $R$ be a symmetric operator on $E$ which is relatively $T$-compact. 
Assume that $T$ and $T+R$ are both invertible. 
Then the difference of positive spectral projections $P_+(T+R) - P_+(T)$ is compact. 
\end{coro}
\begin{proof}
Since $T$ and $T+R$ are invertible, there exists an $\epsilon>0$ such that $(-\epsilon,\epsilon)$ does not intersect with the union $\spec(T) \cup \spec(T+R)$ of the spectra of $T$ and $T+R$. 
Then the positive spectral projections can be defined via continuous functional calculus, i.e., we can take $\chi \in C(\R)$ with $\chi|_{(-\infty,-\epsilon]} \equiv 0$ and $\chi|_{[\epsilon,\infty)} \equiv 1$ and see that $P_+(T) = \chi(T)$ and $P_+(T+R) = \chi(T+R)$. The statement then follows from \cref{prop:cpt_diff_cts_funct}. 
\end{proof}

A regular operator $T$ on $E$ is called \emph{Fredholm} if there exists a \emph{parametrix} $Q$ such that (the closure of) $QT - 1$ and $TQ - 1$ are compact operators on $E$. 
We recall that an odd resp.\ even regular self-adjoint Fredholm operator $T$ on a possibly $\Z_2$-graded Hilbert $A$-module $E$ yields a well-defined class $[T]$ in $\KK^0(\C,A) \simeq \K_0(A)$ resp.\ $\KK^1(\C,A) \simeq \K_1(A)$; for details of the construction, we refer to \cite[\S2.2]{vdD19_Index_DS}. 

Our last result shows that this $\K$-theory class is stable under relatively compact perturbations. 

\begin{prop}
\label{prop:KK-class_rel_cpt_pert}
Let $T$ be a regular self-adjoint Fredholm operator on $E$, and let $R$ be a symmetric operator on $E$ which is relatively $T$-compact. 
Then: 
\begin{enumerate}
\item 
\label{item:rel_cpt_pert_reg_sa}
$T+R$ is also regular, selfadjoint, and Fredholm, and any parametrix for $T$ is also a parametrix for $T+R$. 
\item 
$[T+R] = [T] \in \KK^p(\C,A) \simeq \K_p(A)$ (where $p=0$ if $R,T$ are odd, and $p=1$ otherwise). 
\end{enumerate}
\end{prop}
\begin{proof}
\begin{enumerate}
\item 
We first note that $T+R$ is also regular and selfadjoint by \cref{prop:rel_cpt_pert_reg_sa}. 
If $Q \in \mL_A(E)$ is a parametrix for $T$, then it is also a parametrix for $T+R$, since 
\[
(T+R)Q-1 = (TQ-1) + R(T-i)^{-1} (T-i)Q
\]
is compact. 
Similarly, also $Q(T+R)-1$ is compact. 
\item 
Let $T_t := T + tR$, and consider the operator $T_\bullet = \{T_t\}_{t\in[0,1]}$ on the Hilbert $C([0,1],A)$-module $C([0,1],E)$. 
Since $t \mapsto T_t\psi$ is continuous for each $\psi\in\Dom(T)$, we know that $T_\bullet$ is regular and self-adjoint (\cite[Lemma 1.15]{DM20}). 

If $Q \in \mL_A(E)$ is a parametrix for $T$, then by \ref{item:rel_cpt_pert_reg_sa} it is also a parametrix for $T_t$ for each $t\in[0,1]$, since $tR$ is relatively $T$-compact. Consequently, noting that $t \mapsto tRQ$ is norm-continuous, the constant family $Q_\bullet = \{Q\}_{t\in[0,1]}$ is a parametrix for $T_\bullet$. 
Hence $T_\bullet$ is a regular self-adjoint Fredholm operator on the Hilbert $C([0,1],A)$-module $C([0,1],E)$ and therefore a \emph{homotopy} between $T$ and $T+R$ (in the sense of \cite[Definition 2.13]{vdD19_Index_DS}). Thus $[T]=[T+R]$ by \cite[Proposition 2.14]{vdD19_Index_DS}. 
\qedhere 
\end{enumerate}
\end{proof}

In the $\Z_2$-graded case, where we have a decomposition $E = E_+ \oplus E_-$ and $T$ is odd (i.e., maps $E_\pm \to E_\mp$), the class $[T] \in \KK^0(\C,A)$ corresponds to the $K_0(A)$-valued index of $T_+ := T|_{E_+} \colon E_+ \to E_-$ under the isomorphism $\KK^0(\C,A) \simeq \K_0(A)$. Thus, in this case, the above result translates into the stability of the $\K_0(A)$-valued index under relatively compact perturbations.


\providecommand{\noopsort}[1]{}
\providecommand{\bysame}{\leavevmode\hbox to3em{\hrulefill}\thinspace}
\providecommand{\MR}{\relax\ifhmode\unskip\space\fi MR }
\providecommand{\MRhref}[2]{%
  \href{http://www.ams.org/mathscinet-getitem?mr=#1}{#2}
}
\providecommand{\href}[2]{#2}
\providecommand{\doi}[1]{\href{https://doi.org/#1}{doi:#1}}
\providecommand{\doilinktitle}[2]{#1}
\providecommand{\doilinkbooktitle}[2]{\href{https://doi.org/#2}{#1}}
\providecommand{\doilinkjournal}[2]{\href{https://doi.org/#2}{#1}}
\providecommand{\doilinkvynp}[2]{\href{https://doi.org/#2}{#1}}
\providecommand{\eprint}[2]{#1:\href{https://arxiv.org/abs/#2}{#2}}

\end{document}